%% file: Adaptive3.tex
\documentclass[final]{siamart0516}
\input{adaptive_shared}
\ifpdf
\hypersetup{
  pdftitle={\TheTitle},
  pdfauthor={\TheAuthors}
}
\fi

\begin{document}

\maketitle

\begin{abstract}
We introduce a class of adaptive timestepping strategies for
stochastic differential equations with non-Lipschitz drift
coefficients. These strategies work by controlling potential unbounded
growth in  solutions of a numerical scheme due to the drift. We prove
that the Euler-Maruyama scheme with an adaptive timestepping strategy
in this class is strongly convergent. Specific strategies falling into
this class are presented and demonstrated on a selection of numerical
test problems. We observe that this approach is broadly applicable, can provide more
dynamically accurate solutions than a drift-tamed scheme with fixed
stepsize, and can improve MLMC simulations.
\end{abstract}



\date{\today}

\section{Introduction}
We investigate adaptive timestepping for the numerical approximation of a $d$-dimensional stochastic differential equation (SDE) of \Ito type
\begin{align}\label{eq:SDE}
  dX(t)&=f(X(t))dt+g(X(t))dW(t),\quad t>0,\\
  X(0)&\in\mathbb{R}^d,\nonumber
\end{align}
where $W$ is an $m$-dimensional Wiener process and the drift coefficient $f$ is not globally Lipschitz continuous, but rather satisfies a one-sided Lipschitz condition and a polynomial growth condition.

Since it was pointed out in \cite{HJ2011} that the Euler-Maruyama
method fails to converge in the strong sense for such equations, there
has been much interest in tamed numerical methods, the first of which
was presented in \cite{HJK2012} (see \cref{eq:TE} in Section \ref{sec:MathPreLim}). We also refer the reader to the variant presented in \cite{Sotirios2013}, and to the related class of truncated methods which may be found in, for example, \cite{MR3430145}. Generally, speaking, these methods work by enforcing a higher order modification to the drift and (if necessary) diffusion coefficients in order to control unbounded growth permitted by non-globally Lipschitz coefficients. The idea has been extended to higher order schemes \cite{WangGan2013}, to SDEs with L\`evy noise \cite{MR3513865}, and to stochastic partial differential equations (SPDEs) \cite{MR3498982}. 

However, as noted in \cite{MR3129758}, (fully) tamed methods can lead to dynamically inaccurate
results for even moderately small step-sizes, due at least in part to the perturbation of the flow that results from modifying the coefficients. We illustrate this further in \secref{sec:examples} when we show, for example, that the drift-tamed Euler-Maruyama method does not give a good approximation of the period for the stochastic Van der Pol oscillator.

In this article, we propose an alternative approach to the control of growth arising from a non globally-Lipschitz drift coefficient. Rather than modifying the drift directly, we adjust the length of the timestep taken at each iteration in order to control the norm of the drift response. In spirit this idea is closer to the projected Euler and Milstein methods given in \cite{Beyn2016}, where solutions are prevented from leaving a ball, the radius of which is dependent on the step-size. We also point out \cite{Appleby2010}, where adaptive timestepping was used to control solution dynamics, and in particular to preserve the positivity of solutions of the numerical discretisation of nonlinear SDEs. The recent preprint of Fang \& Giles \cite{Giles2017} takes a related approach; see \cref{rem:FangGiles} in Section
\ref{sec:admiss} for a comparative discussion.

Otherwise, adaptive timestepping for SDEs has tended to concentrate on local error control; see for example \cite{Burrage2004,LMS2006,MR3325826,ShardlowTaylor2015}. A serious drawback of using adaptive methods for SDEs is the potential requirement to interpolate the Brownian path in the case that a timestep is rejected. This is not necessary for the method we propose here, as long as the diffusion coefficient satisfies a global Lipschitz condition. 

The structure of the article is as follows. The remainder of the introduction lays out the mathematical framework for the article, and summarises relevant results from the literature. In \secref{sec:adaptive} we describe the Euler-type discretisation with random stepsize that forms the basis of our scheme. We demonstrate how stepsize controls can be motivated, either by ensuring that the discretised drift coefficient responds similarly to that of a scheme which is known to converge strongly (e.g. tamed Euler), or by examining the dynamics of the discrete drift map. Finally, we define a class of admissible timestepping strategies for \eqref{eq:SDE}, provide examples, and state the strong convergence theorem that is our main result.

In \secref{sec:examples} we investigate our methods with two adaptive timestepping strategies and compare their performance to a tamed Euler method with fixed stepsize for eight test problems, illustrating convergence and reporting the details of stepsizes chosen by each strategy. In particular for the stochastic Van der Pol oscillator we see that the fixed-step tamed Euler method consistently underestimates the period but that adaptive methods give a better approximation. We also examine a multi-level Monte Carlo (MLMC) approximation with adaptive timestepping and observe that this approach reduces the variance on each level, leading to fewer realisations and hence reducing the computational cost.
In \secref{sec:proofs} we provide the proof of our main result. Our conclusions and a short discussion of possible future directions for this work are in \secref{sec:concl}.
  
\subsection{Mathematical preliminaries}
\label{sec:MathPreLim}
Consider the $d$-dimensional \Ito-type SDE \eqref{eq:SDE}. For the remainder of the article we let $(\mathcal{F}_t)_{t\geq 0}$ be the natural filtration of $W$. Suppose $f:\mathbb{R}^d\to\mathbb{R}^d$ is continuously differentiable with derivative that grows at most polynomially: for some $c\in(0,\infty)$
\begin{equation}\label{eq:fDer}
\|Df(x)\|\leq c(1+\|x\|^c);
\end{equation}
and satisfies a one-sided Lipschitz condition with constant $\alpha>0$:
\begin{equation}\label{eq:osLcond}
\langle f(x)-f(y),x-y\rangle\leq \alpha\|x-y\|^2.
\end{equation}
Suppose also that $g:\mathbb{R}^d\to\mathbb{R}^{d\times m}$ is continuously differentiable and satisfies a global Lipschitz condition with constant $\kappa>0$:
\begin{equation}\label{eq:gLcond}
\|g(x)-g(y)\|_F\leq \kappa\|x-y\|.
\end{equation}
Under conditions \eqref{eq:fDer}--\eqref{eq:gLcond}, \eqref{eq:SDE}
has a unique strong solution on any interval $[0,T]$, where $T<\infty$
on the filtered probability space
$(\Omega,\mathcal{F},(\mathcal{F}_t)_{t\geq 0},\mathbb{P})$. Moreover the following moment bounds apply over any finite interval $[0,T]$:
\begin{lemma}\label{lem:boundedMomentsSDE}
Let $f,g$ be $C^1$ functions satisfying \eqref{eq:osLcond} and \eqref{eq:gLcond} respectively. Then for each $p>0$ there is $C=C(p,T,X(0))>0$ such that
\begin{equation}\label{eq:SDEmoments}
\expect{\sup_{s\in[0,T]}\|X(s)\|^p} \leq C.
\end{equation}
\end{lemma}
This was proved as Lemma 3.2 in \cite{HMS2002} for $p>2$, which can be extended to include $0<p\leq 2$ via Jensen's inequality.

The following bound is used to develop timestepping strategies in Section \ref{sec:admiss}, and in the proof of our main theorem.
\begin{lemma}
  \label{HJK012_FBound}
  The polynomial bound on the derivative of $f$ given by
  \eqref{eq:fDer} implies  
  \begin{equation}\label{eq:fbound}
    \|f(x)\|\leq c_1\left(1+\|x\|^{(c+1)}\right).
  \end{equation}
where $c_1:=2c+\|f(0)\|$.
\end{lemma}
\begin{proof}
See, for example, Lemma 3.1 in \cite{HJK2012}. 
\end{proof}
The Euler-Maruyama numerical method and the notion of strong convergence may be expressed as follows.
\begin{definition}
Fix $T<\infty$ and $N\in\mathbb{N}$, and define $h=T/N$. The Euler-Maruyama discretisation of \eqref{eq:SDE} over the interval $[0,T]$ with $N$ steps is given by
\begin{align}\label{eq:EM}
X_{n+1}^N&=X_n^N+hf(X_n^N)+g(X_n^N)(W((n+1)h)-W(nh)),\quad n=0,\ldots,N,\\
X_0&=X(0).\nonumber
\end{align}
\end{definition}
\begin{definition}\label{def:strongConvergence}
If there exists $p\in[1,\infty)$ and constants $C_p,\beta>0$ such that
\[
\left(\mathbb{E}\left[\|X(T)-X_N^N\|^p\right]\right)^{1/p}\leq C_p h^{\beta},
\]
then the Euler-Maruyama method given by \eqref{eq:EM} is said to converge strongly with order $\beta$ in $\mathcal{L}_p$ to solutions of \eqref{eq:SDE} over the interval $[0,T]$.
\end{definition}

In the scalar single noise case, Hutzenthaler \& Jentzen~\cite[Theorem 1]{HJ2011}, showed that the Euler-Maruyama method given in \eqref{eq:EM} cannot converge strongly if at least one of the coefficients grows superlinearly. We restate their result here:
\begin{theorem}
Let $d=m=1$, and let $C\geq 1$, $\beta>\alpha>1$ be constants such that
\[
\max\left\{|f(x)|,|g(x)|\right\}\geq \frac{|x|^\beta}{C}\quad\text{and}\quad\min\left\{|f(x)|,|g(x)|\right\}\leq C|x|^\alpha
\]
for all $|x|\geq C$. If the exact solution of \eqref{eq:SDE} satisfies $\mathbb{E}[|X(T)|^p]<\infty$ for some $p\in[1,\infty)$, then
\[
\lim_{N\to\infty}\mathbb{E}\left[|X(T)-X_N^N|^p\right]=\infty\quad \text{and}\quad \lim_{N\to\infty}\left|\mathbb{E}\left[|X(T)|^p\right]-\mathbb{E}\left[|X_N^N|^p\right]\right|=\infty.
\]
\end{theorem}
The drift-tamed Euler--Maruyama method given by
\begin{equation}\label{eq:TE}
Y_{n+1}^N=Y_n^N+\frac{hf(Y_n^N)}{1+h\|f(Y_n^N)\|}+g(Y_n^N)(W((n+1)h)-W(nh)),\quad n=0,\ldots,N,
\end{equation}
was introduced in \cite{HJK2012} to provide an explicit numerical
method that would display strong convergence in circumstances where
the Euler-Maruyama method does not. In fact, strong convergence was
proved under Conditions \eqref{eq:fDer}--\eqref{eq:gLcond}. The
following theorem states two key results from that article: the 
first on boundedness of moments, the second on strong convergence. 

\begin{theorem}\cite{HJK2012} \label{thm:tamedConv}
Let $X(t)$ be a solution of \eqref{eq:SDE}, where $f$ and $g$ satisfy
Conditions \eqref{eq:fDer}--\eqref{eq:gLcond}. Let $\{Y_n^N\}$ be a solution of \eqref{eq:TE}. Then
\begin{equation}\label{eq:momentBounds}
\sup_{n\in\mathbb{N}}\sup_{n\in\{0,1,\ldots,N\}}\mathbb{E}[\|Y_n^N\|^p]<\infty.
\end{equation}
Let $\{\bar Y^N\}$ be a sequence  of continuous time interpolants of the time discrete approximation \eqref{eq:TE}. There exists a family $C_p$, $p\in[1,\infty)$ of real numbers such that
\begin{equation*}\label{eq:strongConv}
\left(\mathbb{E}\left[\sup_{t\in[0,T]}\|X(t)-\bar Y_t^N\|^p\right]\right)^{1/p}\leq C_p h^{1/2},
\end{equation*}
for all $N\in\mathbb{N}$ and all $p\in[1,\infty)$.
\end{theorem}
Higham, Mao \& Stuart \cite{HMS2002} showed that the Euler-Maruyama
scheme \eqref{eq:EM} is strongly convergent in the sense of \cref{def:strongConvergence} if its moments are bounded in the sense of \eqref{eq:momentBounds} in the statement of Theorem \ref{thm:tamedConv}. In Section \ref{sec:adaptive} we show how stepsize control can be used to bound the drift response pathwise, sufficient to ensure strong convergence.

\section{Adaptive timestepping strategies}
\label{sec:adaptive}
\subsection{Euler-type schemes with random timesteps}
Consider the following Euler-type method for \eqref{eq:SDE} over a
random mesh $\{t_n\}_{n\in\mathbb{N}}$ on the interval $[0,T]$ given by 
\begin{equation}\label{eq:SchemeEM}
  Y_{n+1}=Y_n+h_{n+1}f(Y_{n})+g(Y_{n})\left(W(t_{n+1})-W(t_n)\right),\quad Y_0=X_0,\quad n< N,
\end{equation}
where $\{h_n\}_{n\in\mathbb{N}}$ is a sequence of random timesteps,
and $\{t_n:=\sum_{i=1}^{n}h_i\}_{n=1}^N$ with $t_0=0$. The random time
steps $h_{n+1}$ (and the corresponding point on the random mesh $t_{n+1}$) are to be determined by the value of $Y_n$.
\begin{definition}
Suppose that each member of the sequence $\{t_n\}_{n\in\mathbb{N}}$ is an $(\mathcal{F}_t)$-stopping time: i.e. $\{t_n\leq t\}\in\mathcal{F}_t$ for all $t\geq 0$, where $(\mathcal{F}_t)_{t\geq 0}$ is the natural filtration of $W$. We may then define a discrete-time filtration $\{\mathcal{F}_{t_n}\}_{n\in\mathbb{N}}$ by
\[
\mathcal{F}_{t_n}=\{A\in\mathcal{F}\,:\,A\cap\{t_n\leq t\}\in\mathcal{F}_t\},\quad n\in\mathbb{N}.
\]
\end{definition}
\begin{assumption}\label{assum:h}
Suppose that each $h_n$ is $\mathcal{F}_{t_{n-1}}$-measurable, let $N$ be a random integer such that
\[
N:=\max\{n\in\mathbb{N}\,:\,t_{n-1}<T\} \quad \text{and} \quad t_N=T.
\]
In addition let $h_n$ satisfy the following constraint: minimum and maximum stepsizes $h_{\text{min}}$ and $h_{\text{max}}$ are imposed in a fixed ratio $0<\rho\in\mathbb{R}$ so that 
  \begin{equation}\label{eq:hRatio}
    h_{\text{max}}=\rho h_{\text{min}}.
  \end{equation}
\end{assumption}
In \cref{assum:h}, the lower bound $h_{\text{min}}$ ensures that a simulation over the interval $[0,T]$ can be completed in a finite number of timesteps. In the event that at time $t_{n}$ we compute $h_{n+1}=h_{\text{min}}$, we apply a single step of the drift-tamed Euler method \eqref{eq:TE} over a timestep of length
$h=h_{\text{min}}$, rather than \eqref{eq:SchemeEM}. Therefore the adaptive timestepping scheme under investigation in this article is
\begin{multline}\label{eq:Scheme}
  Y_{n+1}=Y_n+h_{n+1}\left[f(Y_{n})\mathcal{I}_{\left\{h_{n+1}>h_{\text{min}}\right\}}+\frac{f(Y_{n})}{1+h_{\text{min}}\|f(Y_n)\|}\mathcal{I}_{\left\{h_{n+1}=
      h_{\text{min}}\right\}}\right]\\+g(Y_{n})\left(W(t_{n+1})-W(t_n)\right),\quad n=0,\ldots,N-1.
\end{multline}
The upper bound $h_{\text{max}}$ prevents stepsizes from becoming too large and allows us to examine the strong convergence of the adaptive method \eqref{eq:Scheme} to solutions of \eqref{eq:SDE} as $h_{\text{max}}\to 0$ (and hence as $h_{\text{min}}\to 0$).

\begin{remark}\label{rem:conditionalMoments}
In \eqref{eq:Scheme}, note that each $W(t_{n+1})-W(t_n)$ is a Wiener
increment taken over a random step of length $h_{n+1}$ which itself
may depend on $Y_n$, and therefore is not necessarily normally
distributed. However, if $h_{n+1}$ is an
$\mathcal{F}_{t_{n}}$-stopping time then $W(t_{n+1})-W(t_n)$ is
$\mathcal{F}_{t_n}$-conditionally normally distributed with, almost surely (a.s.),
$$\mathbb{E}\left[\|W(t_{n+1})-W(t_n)\|\bigg|\mathcal{F}_{t_{n}}\right]=0,\qquad
\mathbb{E}\left[\|W(t_{n+1})-W(t_n)\|^2\bigg|\mathcal{F}_{t_{n}}\right]=h_{n+1}.$$
In practice therefore, we can replace the sequence of Wiener increments with i.i.d. $\mathcal{N}(0,1)$ random variables denoted $\{\xi_n\}_{n=1}^{N}$, scaled at each step by the $\mathcal{F}_{t_n}$-measurable random variable $\sqrt{h_{n+1}}$.
\end{remark}

In Sections \ref{sec:hTaming} and \ref{sec:locdyn} we provide two
motivating discussions, each illustrating how a timestepping strategy
can be designed. The first focuses on the properties of the drift-tamed Euler method \eqref{eq:TE}, the second on the local dynamics of polynomial maps. In Section \ref{sec:admiss} we set out a sufficient set of conditions for such strategies to ensure that solutions of \eqref{eq:Scheme} converge strongly to those of \eqref{eq:SDE}.
\subsection{Stepsize selection via the drift-tamed Euler map}\label{sec:hTaming}
For the stochastic differential equation \eqref{eq:SDE} the explicit Euler and drift-tamed Euler maps associated with the drift coefficient $f$ are
\begin{equation*}
F_h(y)=y+hf(y)\qquad \text{and} \qquad\tilde{F}_h(y)=y+\frac{hf(y)}{1+h\|f(y)\|}
\end{equation*}
respectively. At each timestep we choose $h(y)$ so that 
\begin{equation}\label{eq:tol}
\|F_{h}(y)-\tilde{F}_{h}(y)\| = \frac{h^2\|f(y)\|^2}{1+h\|f(y)\|}  <\varepsilon
\end{equation}
for some tolerance $\varepsilon>0$. Equivalently we  have   $h^2\|f(y)\|^2-\varepsilon h\|f(y)\|-\varepsilon<0$ and so require $h$ such that
$$\frac{\varepsilon-\sqrt{\varepsilon^2+4\varepsilon}}{2\|f(y)\|}<h<\frac{\varepsilon+\sqrt{\varepsilon^2+4\varepsilon}}{2\|f(y)\|}.$$
Since $\varepsilon-\sqrt{\varepsilon^2+4\varepsilon}<0$ for all $\varepsilon>0$ we are left with the requirement that
\begin{equation*}
0<h<\frac{1}{\|f(y)\|}\left[\frac{\varepsilon+\sqrt{\varepsilon^2+4\varepsilon}}{2}\right],
\end{equation*}
for \eqref{eq:tol} to hold. This leads to an adaptive strategy
\begin{equation}\label{eq:stepAdaptTame}
h_{n+1}(Y_n)=\max\left\{h_{\text{min}},\min\left\{h_{\text{max}},\,\frac{1}{\|f(Y_n)\|}\left[\frac{\varepsilon+\sqrt{\varepsilon^2+4\varepsilon}}{2}\right]\right\}\right\}.
\end{equation}
By construction, each term in the sequence
$\{h_{n}\}_{n\in\mathbb{N}}$ is an $\mathcal{F}_{t_{n-1}}$-measurable
random variable, and \cref{assum:h} holds.

\subsection{Stepsize selection via local dynamics}\label{sec:locdyn}
Consider the drift coefficient function
\begin{equation}\label{eq:fpoly}
f(x)=-\gamma x|x|^\nu,\quad x\in\mathbb{R},
\end{equation}
where $\gamma,\nu>0$. The associated Euler map with stepsize $h$ is given by the function
\begin{equation}\label{def:Fh}
F_h(x)=x-h\gamma x|x|^\nu,\quad x\in\mathbb{R}.
\end{equation}
Consider the discrete-time dynamics of the map given by \eqref{def:Fh} (a more detailed analysis may be found in \cite{AKMR}). The difference equation 
\[x_{n+1}=F_h(x_n)
\] 
has a stable equilibrium solution at zero and an unstable two-cycle at $\left\{\pm \sqrt[\nu]{2/(h\gamma)}\right\}$.  So the basin of attraction of the zero solution is $|x_0|<\sqrt[\nu]{2/(h\gamma)}$. For fixed $\gamma$, we can increase the size of the basin of attraction arbitrarily by choosing $h$ sufficiently small. Moreover, the derivatives are
\[
F_h'(x)=\left\{\begin{array}{ll}
1-h\gamma(\nu+1)x^\nu, & x\geq 0,\\
1+h\gamma(\nu+1)x^\nu, & x<0,
\end{array}\right.
\]
and so outside of the basin of attraction, repeated applications of the map induce oscillations that grow rapidly at a rate determined by $\nu$. At each iteration, a stochastic perturbation with non-compact support can move trajectories outside the basin of attraction and into a regime characterised by rapidly growing oscillation.

This suggests an adaptive timestepping strategy motivated by the control of stability. Our approach is as follows. For \eqref{eq:SDE} with drift coefficient given by \eqref{eq:fpoly}, we select each stepsize to be
\[
h_{n+1}=\max\left\{h_{\text{min}},\min\left\{h_{\text{max}},\frac{1}{\gamma |Y_n^N|^\nu}\right\}\right\}.
\]
This ensures that if the solution moves out of the basin of attraction of the unperturbed equation then the stepsize is decreased so that it is included once again. 

This strategy can be extended to equations with a drift coefficient satisfying the polynomial bound 
\begin{equation}\label{cond:driftsuperlin}
\|f(x)\|\geq \|x\|^\beta/C,
\end{equation}
for $C\geq 1$, $\beta>1$ and all $\|x\|\geq C$, by considering the
basin of attraction of the Euler map corresponding to the polynomial
bound on $f$. This suggests the following adaptation strategy:
\begin{equation}\label{eq:stepAdaptBasin}
h_{n+1}=\max\left\{h_{\text{min}},\min\left\{h_{\text{max}},\frac{\delta}{\|Y_n^N\|^{\beta-1}}\right\}\right\|
\end{equation}
for equations with drift satisfying \eqref{cond:driftsuperlin} with $\beta$ an odd integer and an appropriately chosen $\delta\leq h_{\text{max}}$. More generally, if we consider the growth over a single step of a perturbation $v$ 
governed by the linear equation $$v^{new} = (I + h Df)v \quad \text{so that} \quad  h = \frac{v^{new} -v}{Dfv},$$
then the following strategy is indicated: for some $\delta\leq h_{\text{max}}$, let
\begin{equation}
  h_{n+1} = \max\left\{h_{\text{min}},\min\left\{h_{\text{max}},\frac{\delta}{\|Df(Y_n)\|}\right\}\right\}.
  \label{HL2}
\end{equation}\label{eq:NJB}

\subsection{Strong convergence of adaptive timestepping methods}\label{sec:admiss}
We begin by defining a class of timestepping strategies that guarantee the strong convergence of solutions of \eqref{eq:Scheme} to solutions of \eqref{eq:SDE} by ensuring that, at each step of the discretisation, the norm of the drift response has a pathwise linear bound.
\begin{definition}\label{def:hAdmiss}
Let $\{Y_n\}_{n\in\mathbb{N}}$ be a solution of \eqref{eq:Scheme} where $f$ satisfies \eqref{eq:fDer}-\eqref{eq:osLcond} and $g$ satisfies \eqref{eq:gLcond}. We say that $\{h_n\}_{n\in\mathbb{N}}$ is an \emph{admissible timestepping strategy} for \eqref{eq:Scheme} if \cref{assum:h} is satisfied and there exists real non-negative constants $R_1, R_2<\infty$ such that whenever $h_{\text{min}}<h_{n}<h_{\text{max}}$,
\begin{equation}\label{eq:normfBound}
\|f(Y_n)\|^2\leq R_1+R_2\|Y_n\|^2,\quad n=0,\ldots,N-1.
\end{equation}
\end{definition}

In the next Lemma we provide specific examples of admissible timestepping schemes.
\begin{lemma}\label{lem:normfBound}
Let $\{Y_n\}_{n\in\mathbb{N}}$ be a solution of \eqref{eq:Scheme}, let
$\delta\leq h_{\text{max}}$, and let  $c$ be the constant in
\eqref{eq:fbound}. Let $\{h_n\}_{n\in\mathbb{N}}$ be a timestepping strategy that satisfies \cref{assum:h}. $\{h_n\}_{n\in\mathbb{N}}$ is admissible for \eqref{eq:Scheme} if, for each $n=0,\ldots,N-1$, one of the following holds
\begin{itemize}
\item[(i)] $h_{n+1}\leq\delta/(\|f(Y_n)\|)$;
\item[(ii)] $h_{n+1}\leq \delta/(1+\|Y_n\|^{1+c})$;
\item[(iii)] $h_{n+1}\leq\delta\|Y_n\|/(\|f(Y_n)\|)$;
\item[(iv)] $h_{n+1}\leq \delta\|Y_n\|/(1+\|Y_n\|^{1+c})$,
\end{itemize}
whenever $h_{\text{min}}<h_{n}<h_{\text{max}}$.
\end{lemma}
\begin{proof}
For Part (i) we can apply \eqref{eq:hRatio}:
\[
\|f(Y_n)\|^2\leq\left(\frac{\delta}{h_{n+1}}\right)^2\leq\frac{h^2_{\text{max}}}{h^2_{\text{min}}}=\rho^2,
\]
and so  \eqref{eq:normfBound} is satisfied with $R_1=\rho^2$ and $R_2=0$.

For Part (ii), by \eqref{eq:fbound} and \eqref{eq:hRatio} we have
\[
\|f(Y_n)\|^2\leq (2c+\|f(0)\|)^2(1+\|Y_n\|^{1+c})^2\leq (2c+\|f(0)\|)^2\frac{h^2_{\text{max}}}{h^2_{n+1}}\leq(2c+\|f(0)\|)^2 \rho^2.
\]
and so \eqref{eq:normfBound} is satisfied with $R_1=(2c+\|f(0)\|)^2 \rho^2$ and $R_2=0$.

For Parts (iii) and (iv) similar arguments give the bounds $\|f(Y_n)\|^2|\leq \rho^2\|Y_n\|^2$ and $\|f(Y_n)\|^2\leq(2c+\|f(0)\|)^2 \rho^2\|Y_n\|^2$ respectively, so \eqref{eq:normfBound} is satisfied with $R_1=0$,  $R_2=\rho^2$ for Part (iii), and $R_2=(2c+\|f(0)\|)^2 \rho^2$ for Part (iv).
\end{proof}

Our main result shows the strong convergence in $\mathcal{L}_2$ with order $1/2$ of solutions of \eqref{eq:Scheme} to solutions of \eqref{eq:SDE} when $\{h_{n}\}_{n\in\mathbb{N}}$ is an admissible timestepping strategy. 
\begin{theorem}\label{thm:adaptConv}
Let $(X(t))_{t\in[0,T]}$ be a solution of
\eqref{eq:SDE} with initial value $X(0)=X_0$. Let
$\{Y_n\}_{n\in\mathbb{N}}$ be a solution of \eqref{eq:Scheme} with initial value $Y_0=X_0$ and admissible timestepping strategy $\{h_n\}_{n\in\mathbb{N}}$ satisfying the conditions of \cref{def:hAdmiss}. Then 
\begin{equation*}\label{eq:adaptConv}
\left(\mathbb{E}\left[\|X(T)-Y_N\|^2\right]\right)^{1/2}\leq Ch_{\text{max}}^{1/2},
\end{equation*}
for some $C>0$, independent of $h_{\text{max}}$.
\end{theorem}
The proof of \cref{thm:adaptConv} is a modification of a standard Euler-Maruyama convergence argument accounting for the properties of the random sequences $\{t_n\}_{n\in \mathbb{N}}$ and $\{h_n\}_{n\in \mathbb{N}}$, and using \eqref{eq:normfBound} to compensate for the non-Lipschitz drift. It is presented in Section \ref{sec:proofs}.

It is possible to link the notion of admissibility to the strategies
developed 
via taming and local dynamics in Sections
\ref{sec:hTaming} and \ref{sec:locdyn} as follows.
The adaptive timestepping strategy given by \eqref{eq:stepAdaptTame} is admissible for an appropriate choice of tolerance $\varepsilon$: to see this, let  
\begin{equation*}
0\leq \varepsilon < \frac{\hmax^2}{1+\hmax}.
\end{equation*}
Then \eqref{eq:stepAdaptTame} is equivalent to the strategy given in
Part (i) of \cref{lem:normfBound}, with 
\[
\delta:=\frac{\varepsilon+\sqrt{\varepsilon^2+4\varepsilon}}{2}\leq h_{\text{max}}.
\]
We investigate performance of \eqref{eq:stepAdaptTame} numerically in Section \ref{sec:examples}.

The adaptive timestepping strategy given by
\eqref{eq:stepAdaptBasin} in Section \ref{sec:locdyn} is equivalent
to that given in Part (iii) of \cref{lem:normfBound} when the
drift coefficient is precisely the polynomial expression on the right hand side of
\eqref{cond:driftsuperlin}, and \eqref{eq:stepAdaptBasin} is
therefore admissible in that case. For more general drift
coefficients the closest correspondence is with Part (iv) of 
\cref{lem:normfBound}, for which a priori knowledge of the polynomial
bound parameter $c$ is needed; in practice this may be difficult to
determine. The variant given by \eqref{HL2}, which uses the norm of the Jacobian of $f$, is not known to be admissible but neither does it require precise knowledge of $c$, and we investigate it numerically in Section \ref{sec:examples}.
\begin{remark}\label{rem:FangGiles}
In \cite{Giles2017}, an adaptive timestepping strategy is presented which satisfies
\begin{equation}\label{eq:FGadapt}
\langle Y_n,f(Y_n) \rangle+\frac{1}{2}h_{n+1}\|f(Y_n)\|^2\leq\alpha\|Y_n\|^2+\beta,\quad n=0,\ldots,N-1,
\end{equation}
where the one sided linear bound
$\langle x,f(x) \rangle \leq \alpha\|x\|^2+\beta,$
for $\alpha,\beta>0$, has been imposed upon the drift coefficient $f$. With additional upper and lower bounds on each timestep, and the introduction of a convergence parameter $\delta\leq 1$, the authors show that the Euler-Maruyama scheme is strongly convergent with order $1/2$. 

We note that, in Section 3.1 of \cite{Giles2017}, specific timestepping rules are proposed for two scalar equations with drift satisfying a polynomial bound of the form \eqref{cond:driftsuperlin} for large arguments: the stochastic Ginzburg Landau equation and the stochastic Verhulst equation. These rules are consistent with the adaptive timestepping strategy given by \eqref{eq:stepAdaptBasin}.  Similarly, in Section 3.2, two specific timestepping rules for multi-dimensional SDEs are proposed, the first of which, within our framework, corresponds to Part (iii) of \cref{lem:normfBound}. The second of those rules, within our framework, corresponds to
\[
h_{n+1}\leq \delta\frac{\|Y_n\|^2}{\|f(Y_n)\|^2}.
\]
If we suppose that $\delta\leq h_{\text{max}}$ then we have
\[
\|f(Y_n)\|^2\leq \frac{\delta}{h_{n+1}}\|Y_n\|^2\leq\rho\|Y_n\|^2,
\]
which is admissible for \eqref{eq:Scheme}.
\end{remark}

\section{Numerical examples}\label{sec:examples}
In the numerical experiments below we compare two different
adaptive time-stepping strategies for \eqref{eq:Scheme} with the fixed step drift-tamed
Euler-Maruyama scheme \eqref{eq:TE}. 
For the latter we take as the fixed step $h_{\text{mean}}$ the
average of all timesteps $h_n^{(m)}$ over each path and each realisation
$m=0,1,\ldots,M$ so that 
$$\hmean=\frac{1}{M} \sum_{m=1}^M \frac{1}{N^{(m)}}\sum_{n=1}^{N^{(m)}} h_{n}^{(m)}.$$
Thus we are comparing to a fixed step scheme of similar average cost.
We solve \eqref{eq:Scheme} with the taming inspired
adaptive timestepping strategy \eqref{eq:stepAdaptTame} and denote this \EAT. The fixed step
comparison using $h_{\text{mean}}$ computed from \EAT is denoted \EFT.
Similarly, we solve \eqref{eq:Scheme} with the local dynamics inspired
adaptive timestepping scheme \eqref{HL2} which we denote \EALD and the fixed step
comparison is denoted \EFLD. 

\subsection{A stochastic Ginzburg Landau equation}
This equation arises from the theory of superconductivity and takes the form 
\begin{equation}\label{eq:SGLE}
dX(t)=\left(\left(\eta+\frac{1}{2}\sigma^2\right)X(t)-\lambda X(t)^3\right)dt+\sigma G(X(t))dW(t),\quad X(0)=x_0>0,
\end{equation}
for $t\geq 0$, and where $\eta\geq 0$ and $\lambda,\sigma>0$. 
When $G(X)=X$, the explicit form of the solution over $[0,\infty)$,
provided by Kloeden \& Platen~\cite{KP}, is 
\begin{equation}\label{eq:SGLE:sol}
  X(t)=\frac{x_0\exp(\eta t+\sigma W(t))}{\sqrt{1+2x_0^2\lambda\int_{0}^{t}\exp(2\eta s+2\sigma W(s))ds}},\quad t\geq 0.
\end{equation}
We use this exact solution to illustrate numerically the strong
convergence result of \cref{thm:adaptConv},
see \cref{fig:SGLE:RSC}, computing to a final time of $T=2$ with
$100$ realisations. We
compare in \cref{fig:SGLE:RSC} (a) all four methods \EAT, \EFT, \EALD 
and \EFLD and show reference lines of slope $1$ and $1/2$.
Note that the global error of the adaptive methods at time $T$ is
close to that computed with the mean step $\hmean$ by the fixed step method.
In (b) we show comparison of estimated rates of strong convergence and
root mean square error (RMS) error against the CPU time between the
adaptive methods \EAT, \EALD and the fixed step tamed Euler methods
\EFT, \EFLD. We see there is a slight computational overhead in
performing the adaptive step which is expected.
\begin{figure}
  \begin{center}
    (a) \hspace{0.49\textwidth} (b)
    \includegraphics[width=0.49\textwidth]{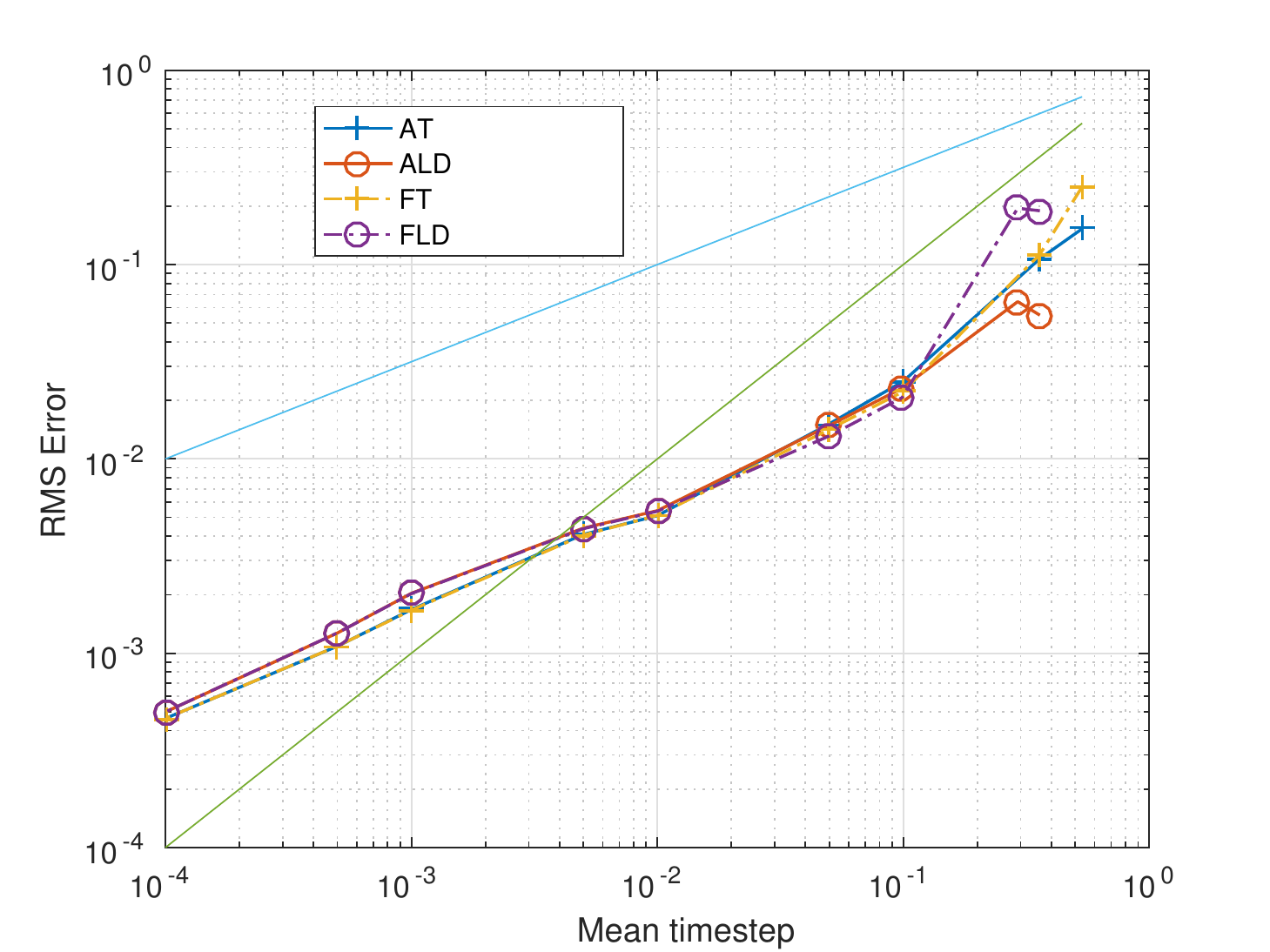}
    \includegraphics[width=0.49\textwidth]{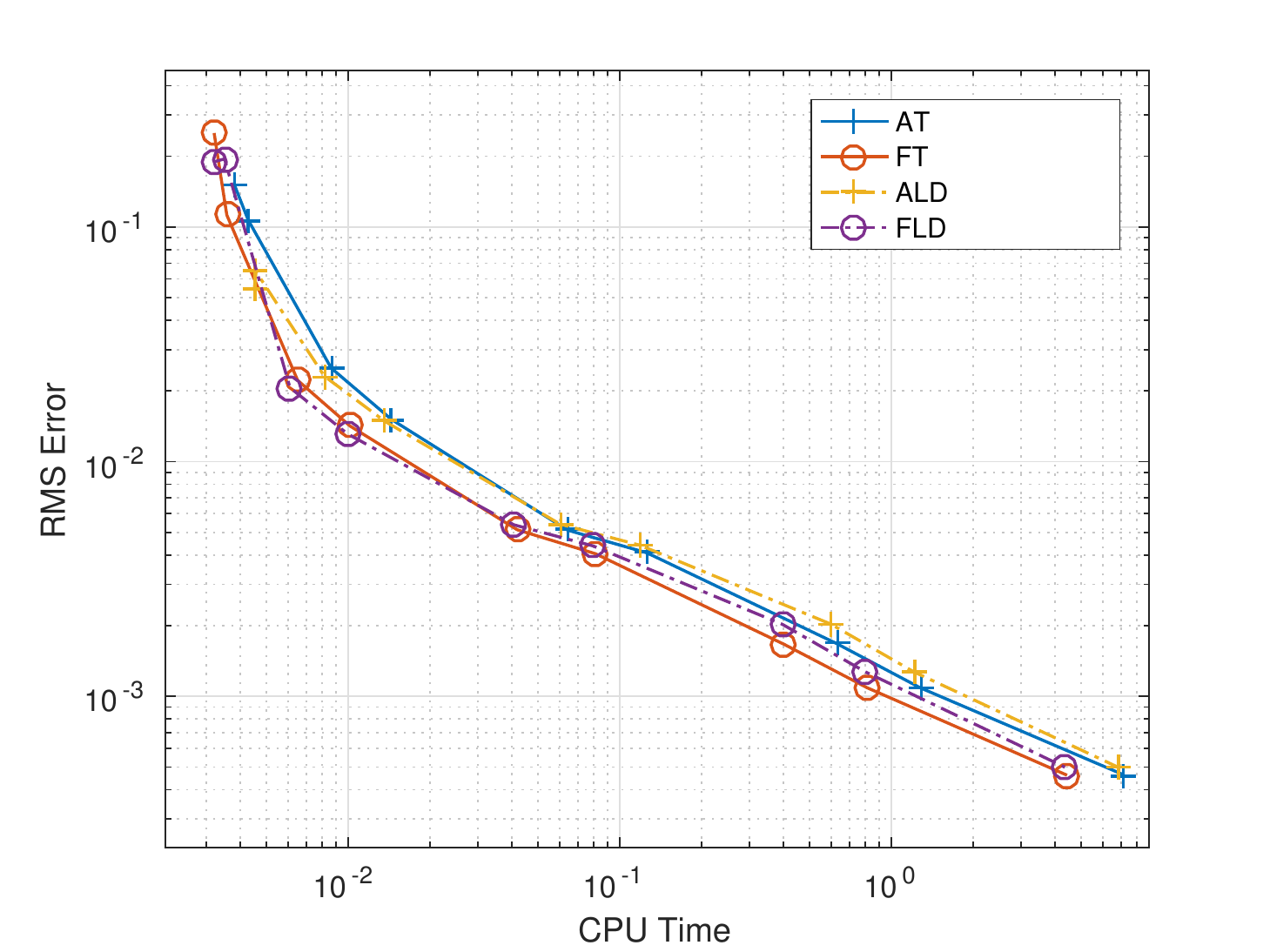}
    \caption{{\bf (a)}: A numerical demonstration of
      strong convergence for multiplicative noise as the mean stepsize decreases for the
      adaptive methods \EAT, \EALD and fixed step methods \EFT, \EFLD
      applied to Eq. \eqref{eq:SGLE}
      with parameters $\eta=0.1$, $\lambda=2$ and $\sigma=0.5$. $T=2$.
      {\bf (b)} plot showing the 
      efficiency and reduction in root mean square (RMS) error as the CPU
      time (s) increases and $\hmax$ decreases. For each $\hmax$ value
      $\rho=100$.} 
    \label{fig:SGLE:RSC} 
  \end{center}
\end{figure}
In \cref{fig:SGLaddConv} we examine convergence for \cref{eq:SGLE}
with additive noise (taking $G(X)=1$). As we do not have an exact
solution we use a reference solution computed with $h=10^{-5}$ using
\eqref{eq:TE}. We observe, as for a standard Euler-Maruyama method, an
improvement in the rate of convergence for the adaptive methods
\EAT, \EALD as well as the fixed step schemes \EFT and \EFLD.
Comparing with $\hmean$ leads to similar errors and we again note a
slight computational overhead to account for the adaptive step in the algorithm.
\begin{figure}
  \begin{center}
    (a) \hspace{0.49\textwidth} (b)
    \includegraphics[width=0.49\textwidth]{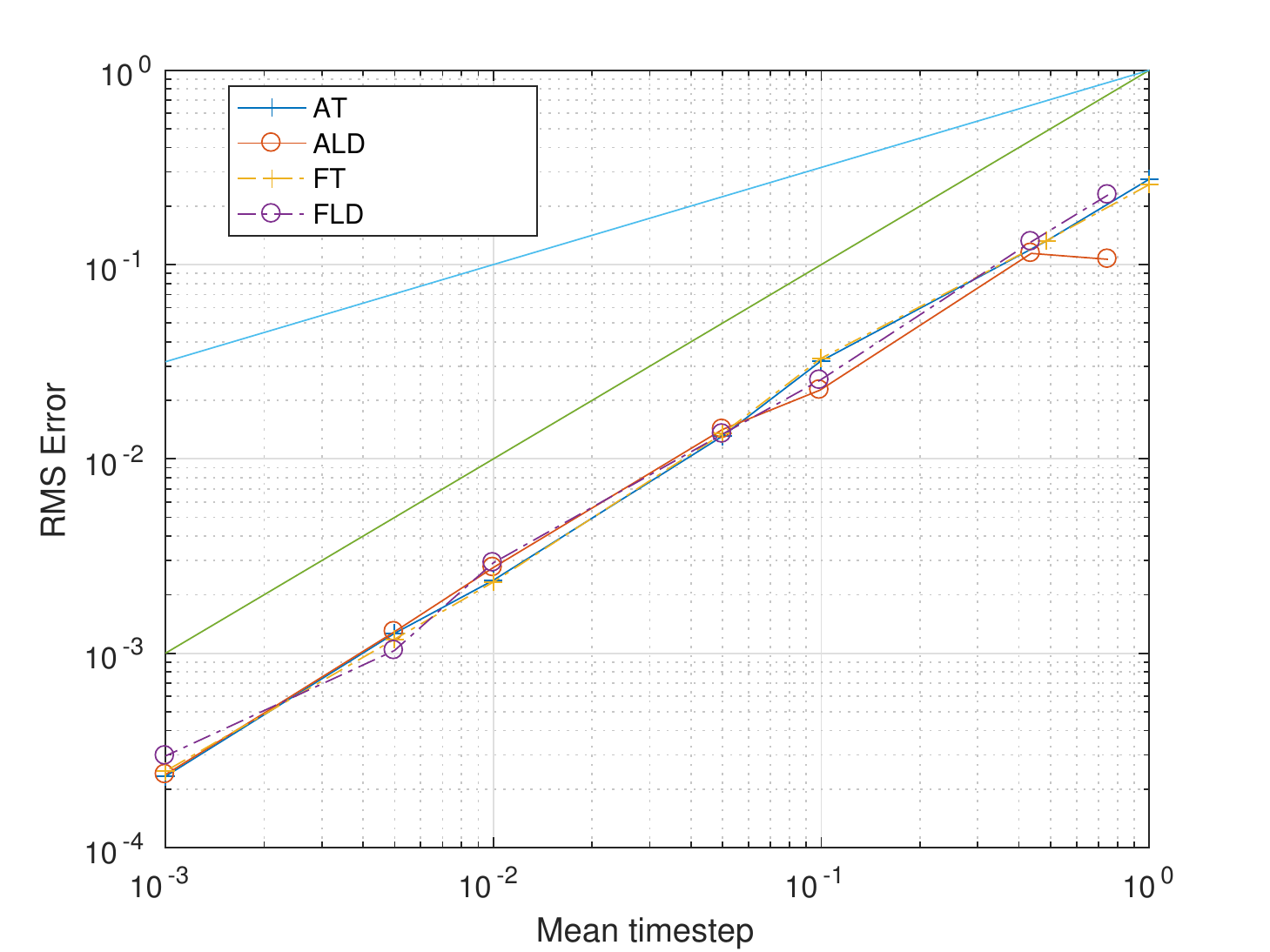}
    \includegraphics[width=0.49\textwidth]{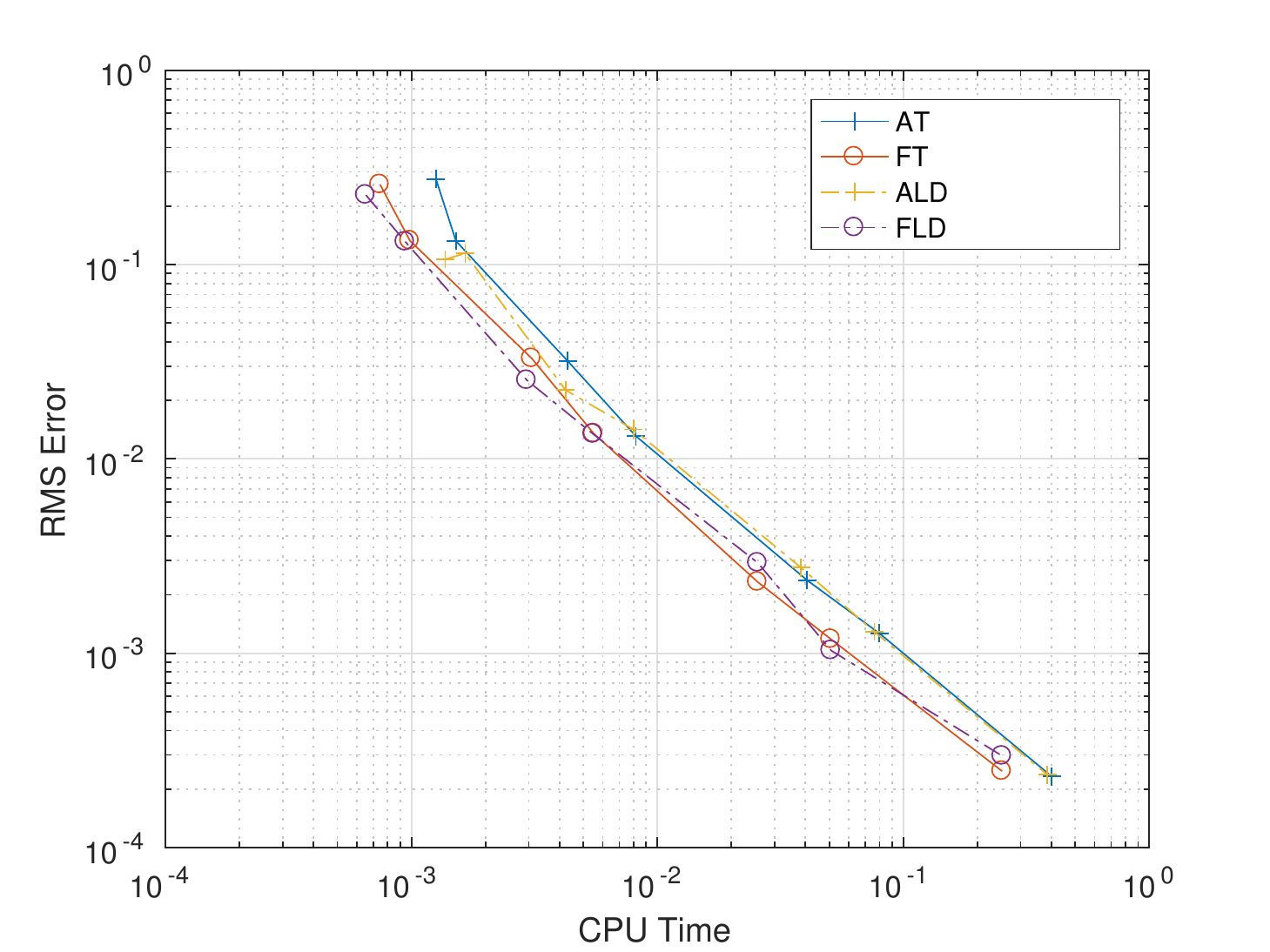}
    \caption{{\bf (a)}: A numerical demonstration of
      strong convergence  for addititive noise as the mean stepsize decreases for the
      adaptive methods \EAT, \EALD and fixed step methods \EFT, \EFLD
      applied to Eq. \eqref{eq:SGLE}
      with parameters $\eta=0.1$, $\lambda=2$ and $\sigma=0.5$. $T=2$.
      {\bf (b)} plot showing the 
      efficiency and reduction in root mean square (RMS) error as the CPU
      time (s) increases and $\hmax$ decreases. For each $\hmax$ value
      $\rho=100$.} 
    \label{fig:SGLaddConv} 
  \end{center}
\end{figure}

\subsection{The stochastic Van der Pol oscillator}
\label{sec:vdpol}
This is a stochastic additive noise version of the van der Pol
oscillator, which describes the effect of external noise on stable
oscillations, and takes the form 
\begin{equation}\label{eq:SVDP}
  d\begin{pmatrix}X_1(t)\\X_2(t)
  \end{pmatrix}=
  \begin{pmatrix}
    X_2(t)\\
    (1-(X_1(t))^2)X_2(t)- X_1(t)
  \end{pmatrix}dt
  +\begin{pmatrix}
    0\\
    dW(t)
  \end{pmatrix}.
\end{equation}
In \cref{fig:vdpolPER_TRUE} we show two realisations for
\eqref{eq:SVDP} obtained using the drift-tamed Euler--Maruyama scheme
\eqref{eq:TE} with $h=10^{-4}$. We clearly see periodic behaviour over
the interval $[0,T]$. We ask how well the period is captured by
the adaptive methods \EAT and \EALD and by the fixed step methods \EFT
and \EFLD. \cref{fig:vdpolPER_COMP} compares two 
realisations computed using the same paths for $W(t)$, so that the
path in (a) is the same as that in (b) (similarly for (c) and (d). We observe
that the fixed step methods \EFT and \EFLD in (b) and (d) appear to
have fewer oscillations than the adaptive simulations in (a) and (c)
(and \cref{fig:vdpolPER_TRUE} (a)).
\begin{figure}
  \begin{center}
    (a) \hspace{0.49\textwidth} (b)
    \includegraphics[width=0.48\textwidth]{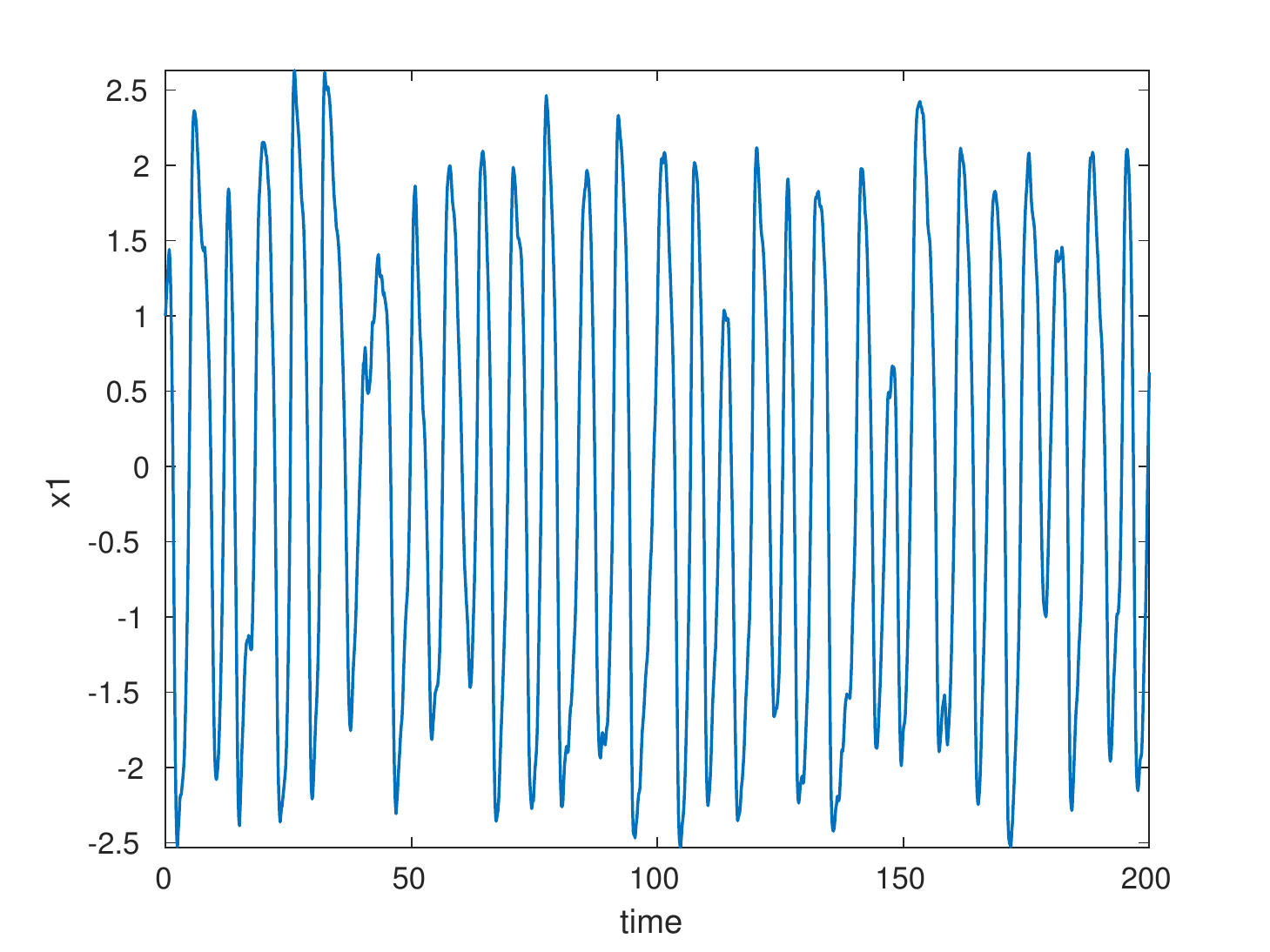}
    \includegraphics[width=0.48\textwidth]{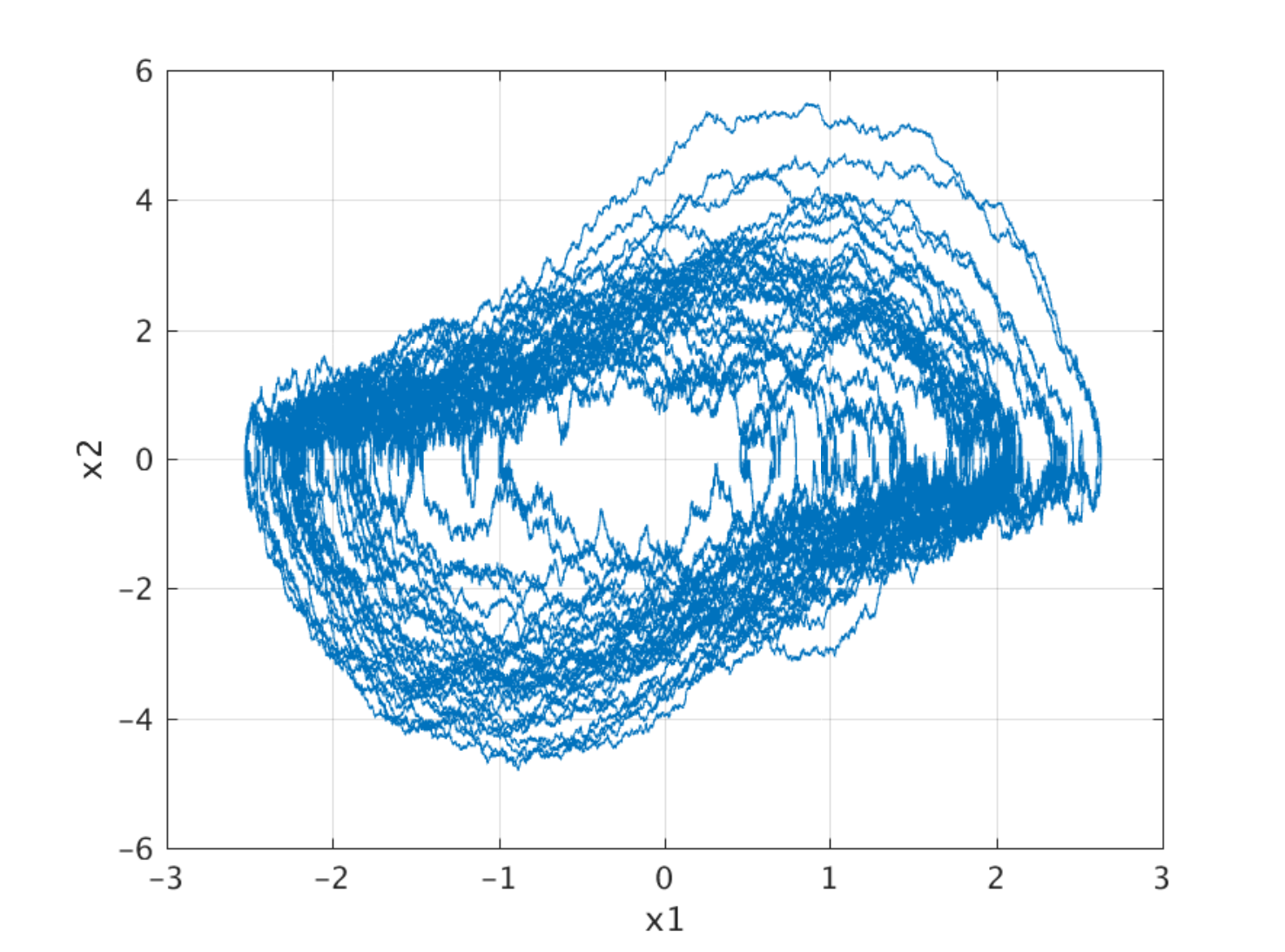}
    \caption{A single realisation of \eqref{eq:SVDP} obtained using
      \eqref{eq:TE} with $h=10^{-4}$. (a) shows $X_1(t)$ against time
      and (b) the phase portrait.} 
    \label{fig:vdpolPER_TRUE}
  \end{center}
\end{figure}

\begin{figure}
  \begin{center}
    (a) \hspace{0.49\textwidth} (b)
    \includegraphics[width=0.48\textwidth]{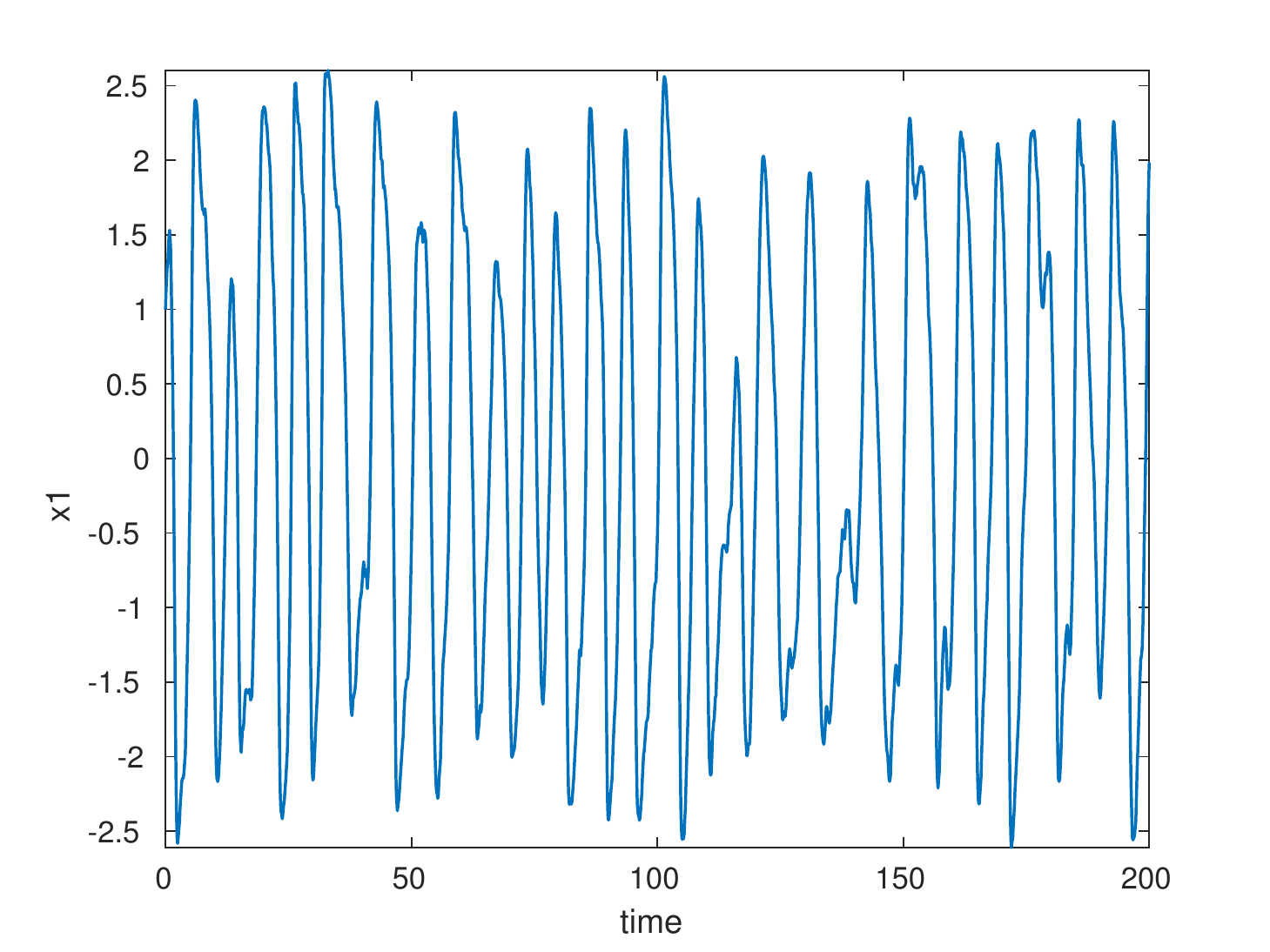}
    \includegraphics[width=0.48\textwidth]{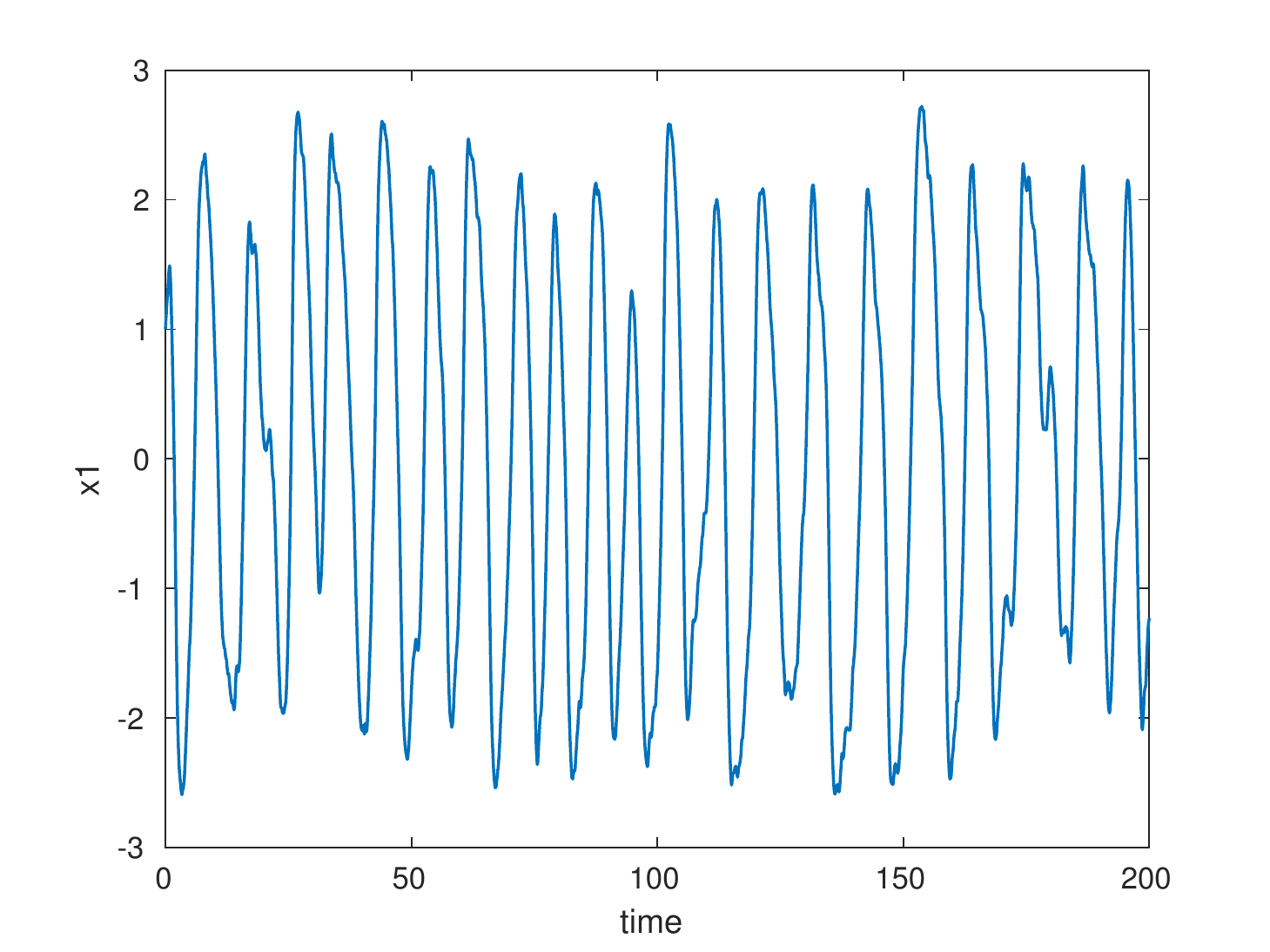}\\
    (c) \hspace{0.49\textwidth} (d)
    \includegraphics[width=0.48\textwidth]{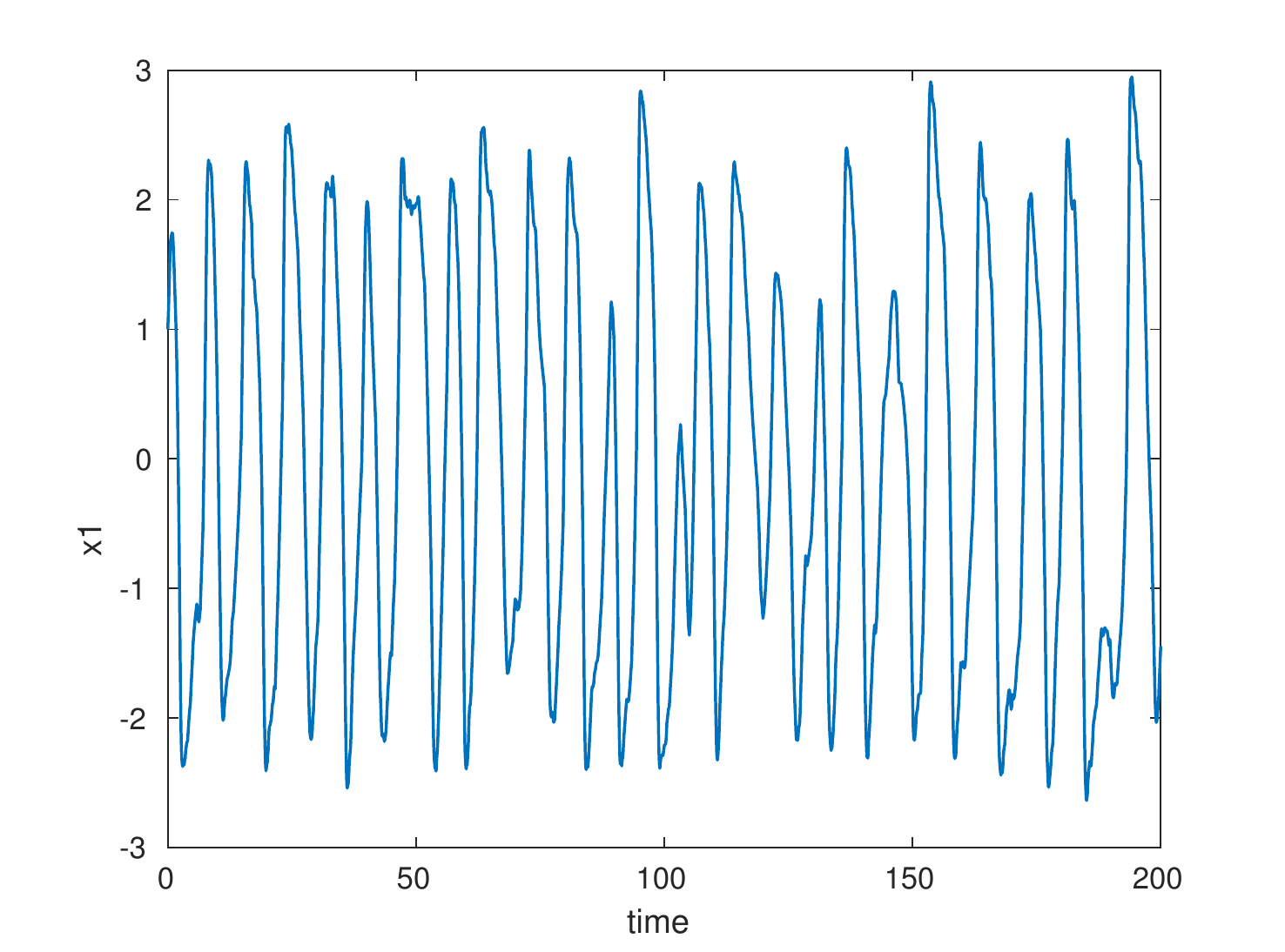}
    \includegraphics[width=0.48\textwidth]{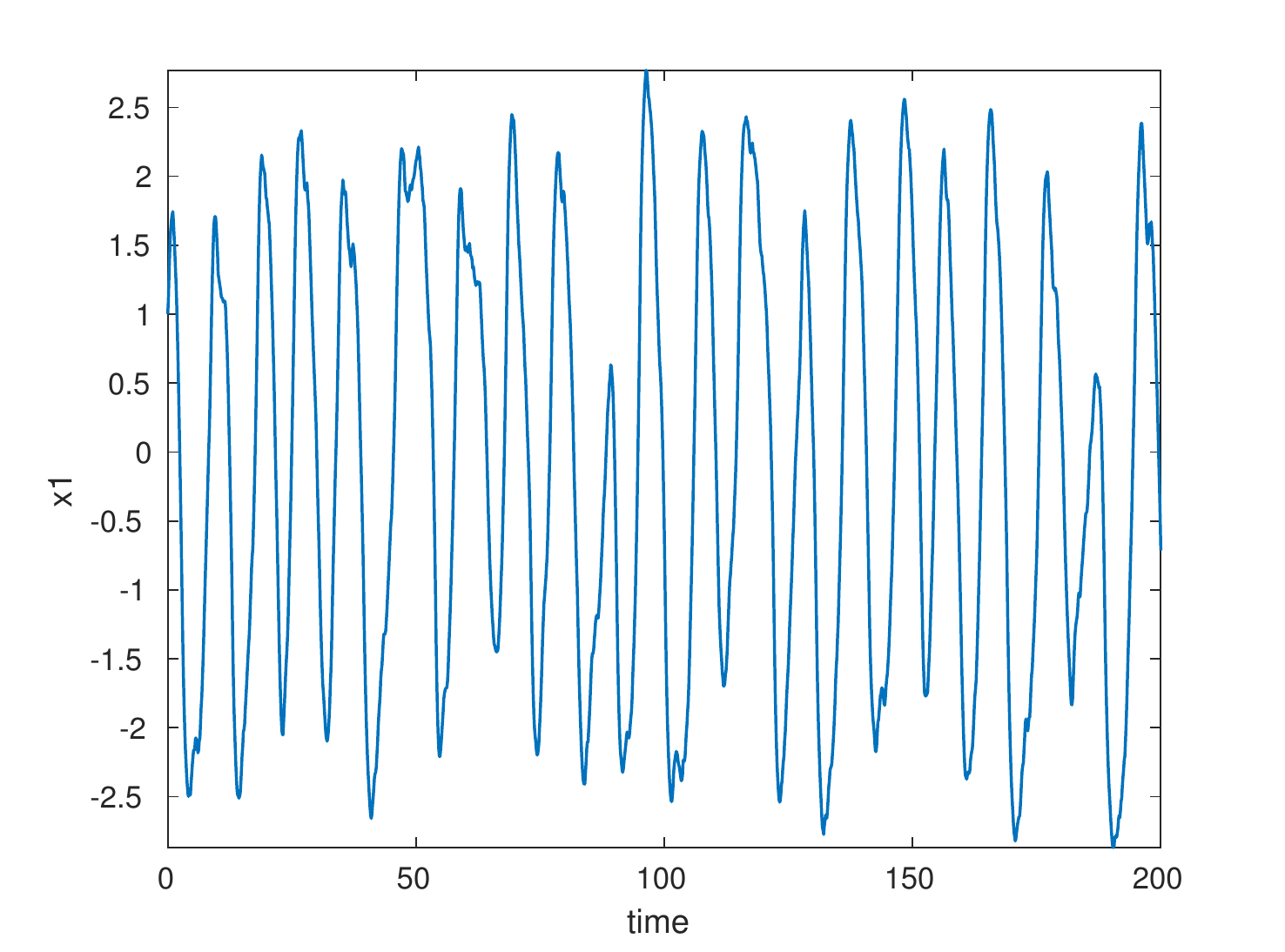}
    \caption{Sample realisations to \eqref{eq:SVDP} obtained using \EAT (a)
      compared to $h=0.0838$ for \EFT in (b). In (c) we use \EALD and
      compare to \EFLD with $h=0.1269$ in (d). Here $\rho=100$ and
      $\hmax=1$. Note that the paths in (a) and (b) (and (c) and (d))
      are the same and the difference arises from the timestepping.}
    \label{fig:vdpolPER_COMP}
  \end{center}
\end{figure}
This is borne out in \cref{tab:periods} which compares data on
the estimated mean period and variance from $100$
realisations of \eqref{eq:SVDP} for $t\in[0,100]$. We also include maximum and
minimum periods observed.  The adaptive methods \EAT and \EALD both give a
better estimate of the period than the equivalent fixed step methods
and have a smaller relative error. 
We also note that \EAT uses, on average, smaller steps than \EALD and
has a smaller relative error. For the equivalent fixed step schemes
\EFT and \EFLD the error for these different timesteps are similar.

\begin{table}
  \begin{tabular}{ | l | c | c |c |c|c|c|}    \hline 
  & Rel. Error & Mean Period & Var &  Min & Max  & h 
 \\ \hline 
TE \eqref{eq:TE} & & 6.684832 & 0.294930 & 5.555556 & 9.090909 & 0.0005 \\ 
\EAT &
0.089037 & 7.355539 & 0.484346 & 5.882353 & 9.090909 &  \\ 
\EFT &
0.213953 & 8.543185 & 0.808813 & 7.142857 & 11.111111 & 0.080635 \\ 
\hline 
TE \eqref{eq:TE}& & 6.725343 & 0.250395 & 5.555556 & 8.333333 & 0.0005 \\ 
\EALD & 
0.183946 & 8.368017 & 1.750757 & 6.250000 & 14.285714 & \\ 
\EFLD & 
0.279599 & 9.394958 & 1.132636 & 7.142857 & 14.285714 & 0.120965 \\ 
\hline 
  \end{tabular}
  \caption{Comparison for the van der Pol equation \eqref{eq:SVDP} of
    estimated mean period, 
    variance, minimum period and maximum period based on $100$
    realisations with $\rho=100$, $\hmax=1$ and $T=100$. We also report
    an estimate of the relative error in the mean period. 
  } 
  \label{tab:periods}
\end{table}
In \cref{tab:stepsVDPOL} we examine for $T=200$ the timesteps $h_n$
taken by \EAT and \EALD for different values of $\rho$ with $\hmax=2$.
We report $h_{\text{mean}}$, along with the timestep variance, the minimum and maximum timesteps,
the computational time taken, and the percentage of timesteps taken
at the minimum $\hmin$. We see that for $\rho$ large enough $\hmin$ is not reached often and the frequency with which this occurs for $\rho=100$ (where $\hmin=0.02$) is
similar to that for $\rho=1000$ (where $\hmin=0.002$). 
\begin{table}
  \begin{tabular}{|l|c|c|c|c|c|c|c|} \hline 
 & $\rho$ & $\hmean$ & Var $h_n^{(m)}$  &  Min $h_n^{(m)}$ & Max $h_n^{(m)}$  & cpu (s)& \% Min\\ 
 \hline 
\EAT & $1000$ &
0.080644 & 0.004299 & 0.012224 & 0.823429 & 1.068599 &0.000000 \\ 
\EAT& $100$ &
0.081051 & 0.004195 & 0.020000 & 0.798538 & 1.071684 &6.064423 \\ 
\EAT& $10$ &
0.207301 & 0.001526 & 0.200000 & 0.820070 & 0.461500 &89.625417 \\ 
\hline 
\EALD& $1000$ &
0.122547 & 0.008976 & 0.022985 & 0.499999 & 0.663723 &0.000000 \\ 
\EALD& $100$ &
0.121729 & 0.008901 & 0.023002 & 0.499999 & 0.671226 &0.087541 \\ 
\EALD& $10$ &
0.220997 & 0.004053 & 0.200000 & 0.499998 & 0.413096 &83.616568 \\ 
\hline 
  \end{tabular}
  \caption{Comparison of step size data $h_n$ for the stochastic Van der Pol
    equation \eqref{eq:SVDP} with additive noise based on $100$
    realisations.
    See \cref{tab:stepsLang} for an example with multiplicative noise.} 
  \label{tab:stepsVDPOL}
\end{table}

\subsection{A Langevin equation}
The following example is taken from \cite{LMS2006}:
\begin{equation}\label{eq:LE}
  \begin{split}
    dX_1(t)&=X_2(t)dt\\
    dX_2(t)&=-\left[\frac{1}{2}X_2(t)\left(\frac{4(5X_1(t)^2+1)}{5(X_1(t)^2+1)}\right)^2\right]dt+\frac{4(5X_1(t)^2+1)}{5(X_1(t)^2+1)}dW(t).
  \end{split}
\end{equation}
We take $X(0)=[1,1]^T$ and solve to $T=20$ with $\hmax=2$. We now
examine the choice of $\rho$. In \cref{tab:stepsLang} we give the
mean step $\hmean$, variance, minimum and maximum step, computational time
and the percentage of steps that were at $\hmin$. Note that both
$\hmax$ and $\hmean$ is larger for \EAT than \EALD (and we see a smaller
computational time). 
In \cref{fig:langPercent} we plot the percentage of the number of
steps taken at the minimum step size as $\rho$ is increased for \EAT
and \EALD. We see that for small $\rho$ the minimum step $\hmin$ is
reached with a high probability ($1$ when $\rho=1$). As $\rho$ is
increased for both schemes the minimum step is no longer reached (at
$\rho=10^{3}$ for \eqref{HL2} and $\rho=10^{4}$ for  \eqref{eq:tol}).
This illustrates that the time adaptivity is actively controlling the
dynamics (and we are not at the minimum step at each iteration).
Although from  \cref{fig:langPercent} we can not see that \EAT or
\EALD takes larger or smaller steps we can see that the step size
choice is different and that the variance is smaller for \EALD
(and in some situations it may be preferable not to have large
switches in stepsize).
\begin{table}
  \begin{tabular}{|l|c|c|c|c|c|c|c|}    \hline 
 & $\rho$ & $\hmean$ & Var $h_n^{(m)}$  &  Min $h_n^{(m)}$ & Max
                                                                 $h_n^{(m)}$
    & cpu & \% Min\\ 
 \hline 
\EAT & $1000$ &
0.084782 & 0.025490 & 0.011188 & 1.487107 & 0.104582 &0.000000 \\ 
\EAT& $100$ &
0.085909 & 0.024528 & 0.020000 & 1.396040 & 0.114417 &19.373834 \\ 
\EAT& $10$ &
0.237467 & 0.023662 & 0.200000 & 1.118137 & 0.040255 &87.013139 \\ 
\hline 
\EALD& $1000$ &
0.012929 & 0.001416 & 0.002009 & 0.464278 & 0.790780 &12.960151 \\ 
\EALD& $100$ &
0.035592 & 0.003088 & 0.020000 & 0.468432 & 0.288227 &74.265358 \\ 
\EALD& $10$ &
0.212098 & 0.002462 & 0.200000 & 0.446987 & 0.046083 &93.070037 \\ 
\hline 
  \end{tabular}
  \caption{Comparison of step size data $h_n$ for the Langevin
    equation \eqref{eq:LE} with multiplicative noise based on $100$
    realisations.
    See \cref{tab:stepsVDPOL} for an example with additive noise.} 
  \label{tab:stepsLang}
\end{table}

\begin{figure}
  \begin{center}
  \includegraphics[width=1.0\textwidth,height=0.175\textheight]{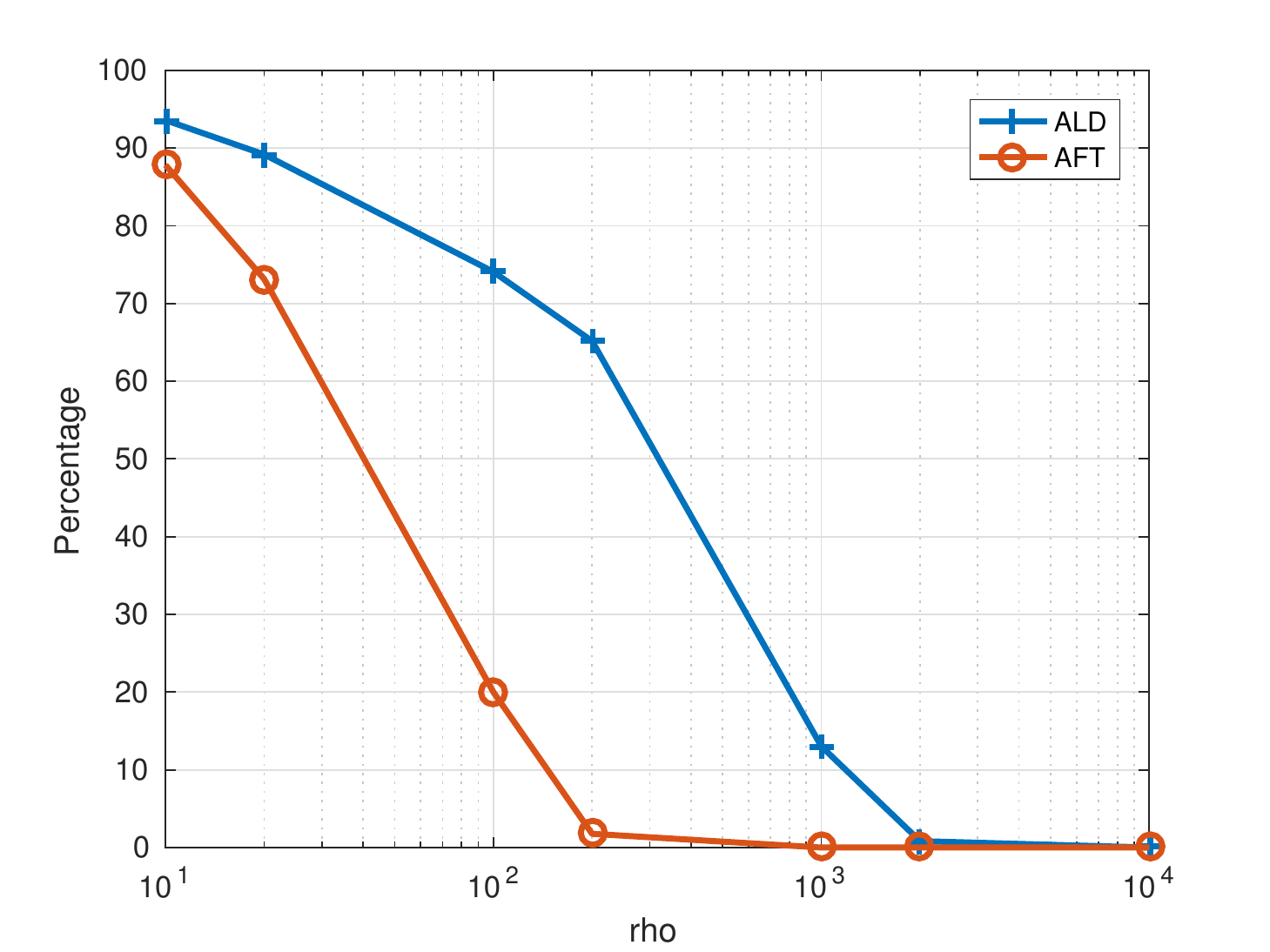}
  \caption{As $\rho$ increases, the percentage of steps taken at the
    smallest step $\hmin$ decreases for \EAT and \EALD. Here $\hmax=2$
  and so $\hmin=0.2,\ldots,0.0002$.}
\label{fig:langPercent}
\end{center}
\end{figure}

\subsection{Comparison of step sizes for 8 different problems}
Both \EAT and \EALD control growth from a non globally Lipschitz drift term.
We now compare the timestep selection made by each of these over a
range of different problems. Rather than present tables on data such
as in \cref{tab:stepsVDPOL} or \ref{tab:stepsLang} we summarize
the mean and variance in \cref{fig:8problems} for $\rho=100$ (a) and
$\rho=1000$ in (b). We include the stochastic Van der Pol oscillator
\eqref{eq:SVDP} (VdP), the Langevin equation \eqref{eq:LE} 
(Lang), and the Stochastic Ginzburg-Landau equation \eqref{eq:SGLE}
with $G(X)=1$ (additive noise) (SGLA).
The other models that we examine can also be found, for example, in
\cite{HJ2013}. Note that for certain of these models the
coefficients change randomly on each realisation.
\begin{description}
\item[SIR:] Simulation of the stochastic Susceptible, Infected,
  Recovered (SIR) model
  \begin{align*}
    dX_1(t) = &[-\alpha X_1(t) X_2(t)-\delta X_1(t)+\delta ]dt + [-\beta X_1(t)X_2(t)] dW_1(t),\\
    dX_2(t) = &[\alpha X_1(t) X_2(t)-(\gamma+\delta) X_2(t)]dt + [\beta X_1(t)X_2(t)] dW_2(t),\\
    dX_3(t) = &[\gamma X_2(t)-\delta X_3(t)]dt,
  \end{align*}
  over the simulation interval [0,T], $T=2$ with initial data $X(0) =
  [0.5;0.3;0.2]$. 
  For each simulation we take $\alpha, \beta, \gamma, \delta \sim
  U[0,10]$.
\item[LV:] Simulation of the stochastic Lokta-Volterra (LV) model in the well stirred sense
\begin{align*}
  dX_1(t) = & [X_1(t)  (\alpha - \beta  X_2(t))]dt + \sigma_1  X_1(t)  dW_1(t),\\
  dX_2(t) = & [X_2(t)  (\gamma  X_1(t)-\delta)]dt + \sigma_2  X_2(t)  dW_2(t),
\end{align*}
over the simulation interval [0,T] with $T=20$ and initial value
$X(0)=[5,10]^T$. 
The parameters $\alpha, \beta, \gamma, \delta \sim U[0,1]$ for each
realisation and $\sigma_1=\sigma_2=0.01$.

\item[PK:] Simulation of a Proto-Kinetics (PK) model.
Here $X$ represents the proportion of one form of a certain 
protein and therefore should be constrained to the interval
$[0,1]$ and can be modelled by the following SDE
\begin{align*}
  dX(t) = & \left[\frac{1}{2}-X(t)+ X(t)(1 - X(t))+\frac{1}{2}
            X(t)(1-X(t))(1-2 X(t))\right]dt  \nonumber \\ 
          & + [X(t)(1-X(t))] dW(t).
\end{align*}
We take the simulation interval to be $[0,T]$, $T=100$.

\item[2D:] Simulations of the polynomial type SDE
  $$ dX(t) = (A X(t)-\beta X(t)|X(t)|^{\nu}) dt + G dW(t),$$
  where $A, \beta, G\in \real^{J\times J}$ $X(0) = [-1 ; -1]^T$.
  We take $\nu=2$ and
  $$A=\left(\begin{matrix}0.807019 & 0.589848 \\  0.080506 &
      0.477723\end{matrix}\right), \,\,
  \beta=\left(\begin{matrix}0.99133 &   0.60672 \\  0.29234 &
      0.96434\end{matrix}\right), \,\,
  G=\left(\begin{matrix}0.5 &   0\\  0 &  0.5\end{matrix}\right).
  $$
 
 \item[CIR:] Simulation of special case of the stochastic
  Cox-Ingersoll-Ross (CIR) model
  $$ X(t) = \kappa(\theta-X)dt+ \sigma \sqrt{|X|} dW, \qquad X(0)=1$$
  over $[0,T]$, $T=200$ with $\kappa=0.1$, $\theta=0.5$ and
  $\sigma=0.5$. 

\end{description}

\begin{figure}
  \begin{center}
    (a) \hspace{0.49\textwidth} (b)
    \includegraphics[width=0.49\textwidth]{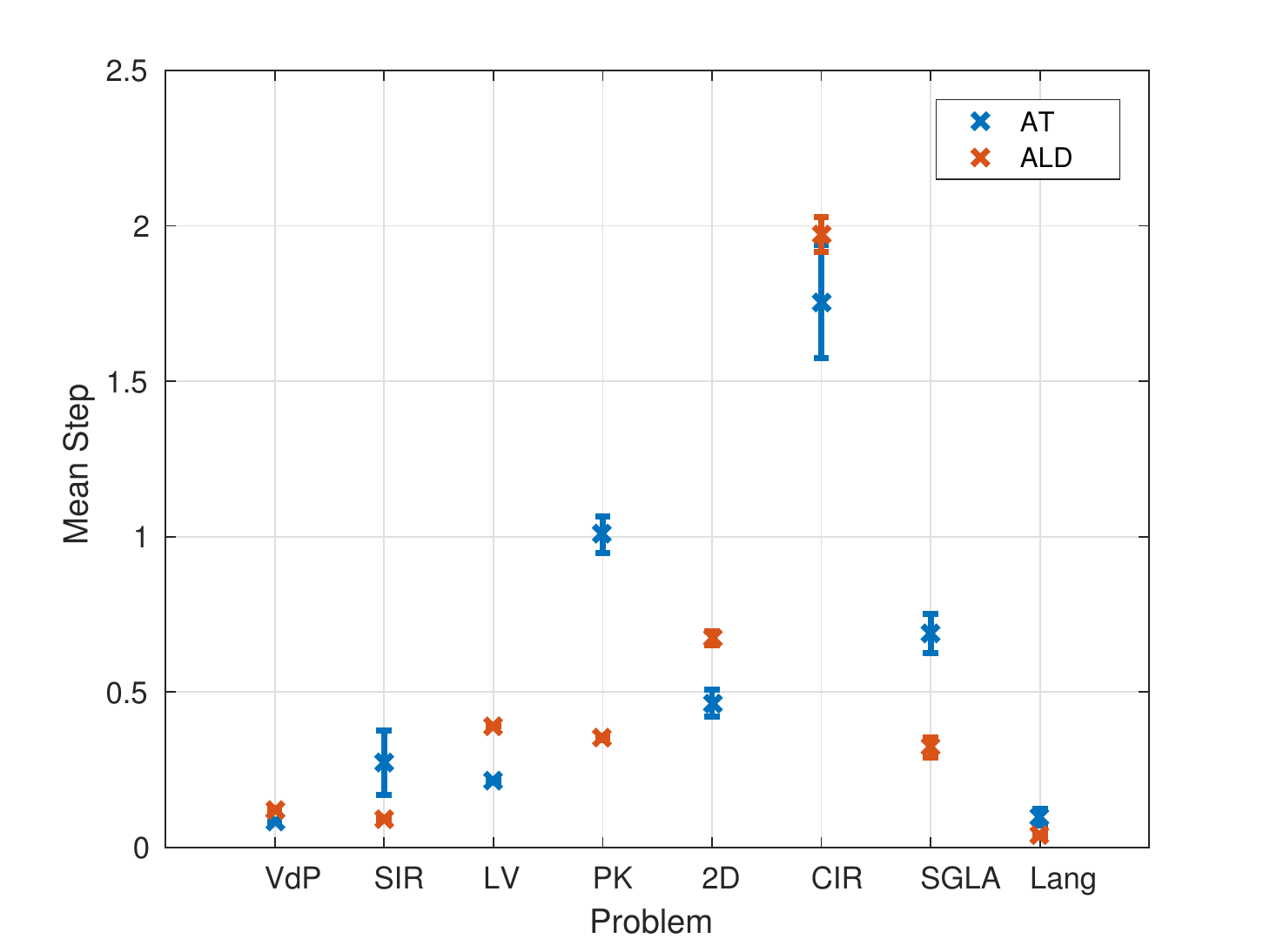}
    \includegraphics[width=0.49\textwidth]{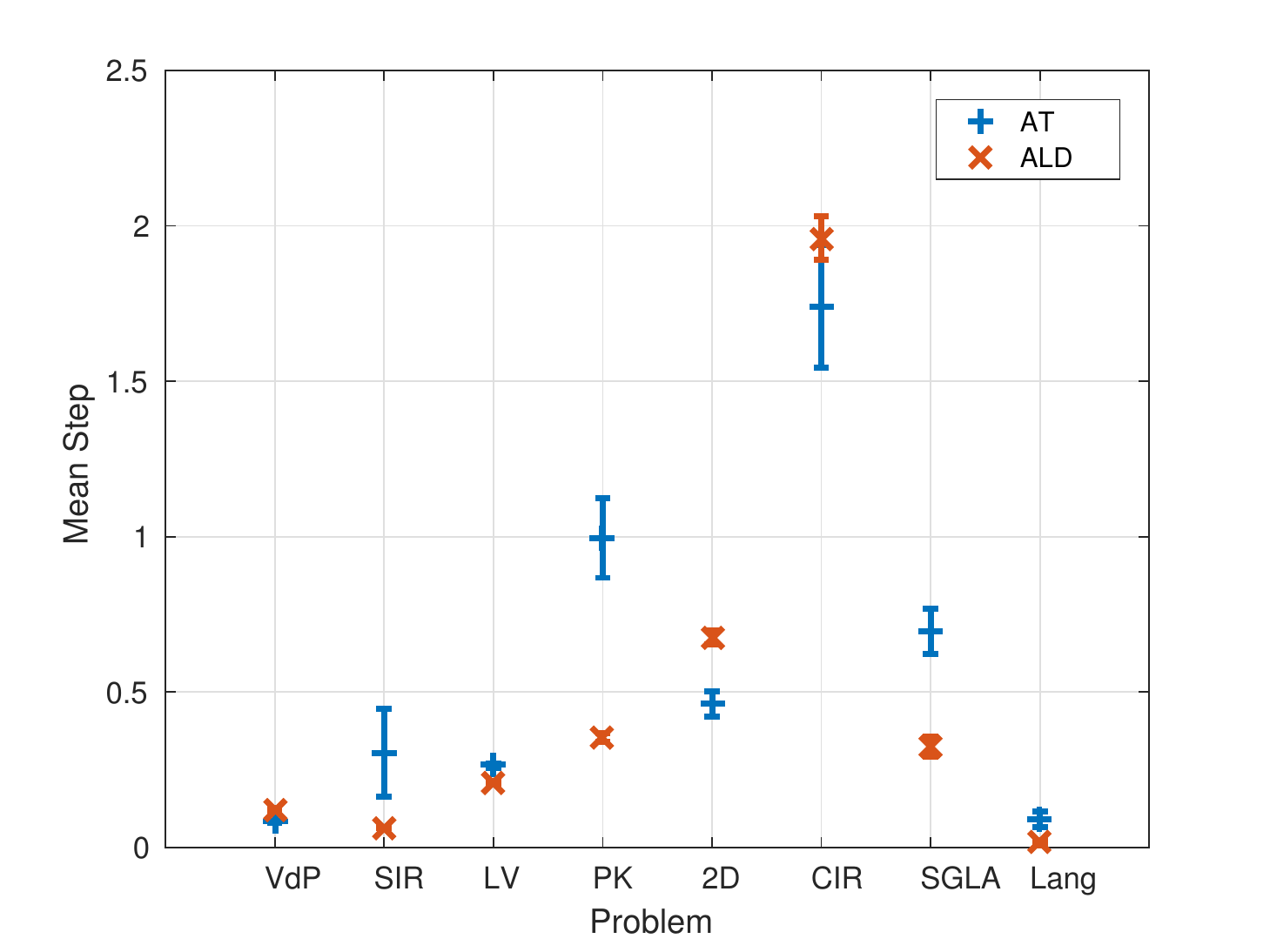}
    \caption{Comparison of mean step sizes $\hmean$ for 8 different problems with $\hmax=1$
      {\bf (a)} $\rho=100$ and {\bf (b)} $\rho=1000$.}
  \end{center}
  \label{fig:8problems}
\end{figure}

\subsection{Stochastic Allen-Cahn SPDE}
To investigate adaptive timestepping for a large system of SDEs we consider the
discretisation of the Allen-Cahn SPDE 
$$du = \left[D u_{xx} + u-u^3 \right] dt + \sigma dW,$$
with $x\in [0,1]$, periodic boundary conditions, and initial data
$u(0,x)=\sin(2\pi x)$. The $Q-$Wiener process $W$ is white in time
and takes values in $H^1_{per}(0,1)$. We take $D=0.01$ and
$\sigma=0.5$ and discretise in space by a spectral Galerkin
approximation \cite{LPS} to get an SDE system in $\mathbb{R}^{100}$.
We take $\hmax=0.05$ and $\rho=100$.
We show in \cref{fig:ACSpde} (a) the $L^2(0,1)$ norm of one sample
realisation as we solve over $t\in[0,10]$ using \EAT and in (b) we
plot the corresponding timestep $h_n$. Note that where the $L^2(0,1)$ of
the solution becomes small in (a), and hence the non-linearity becomes
small, larger steps are taken.
\begin{figure}
  \begin{center}
    (a) 
    \includegraphics[width=1\textwidth,height=0.15\textheight]{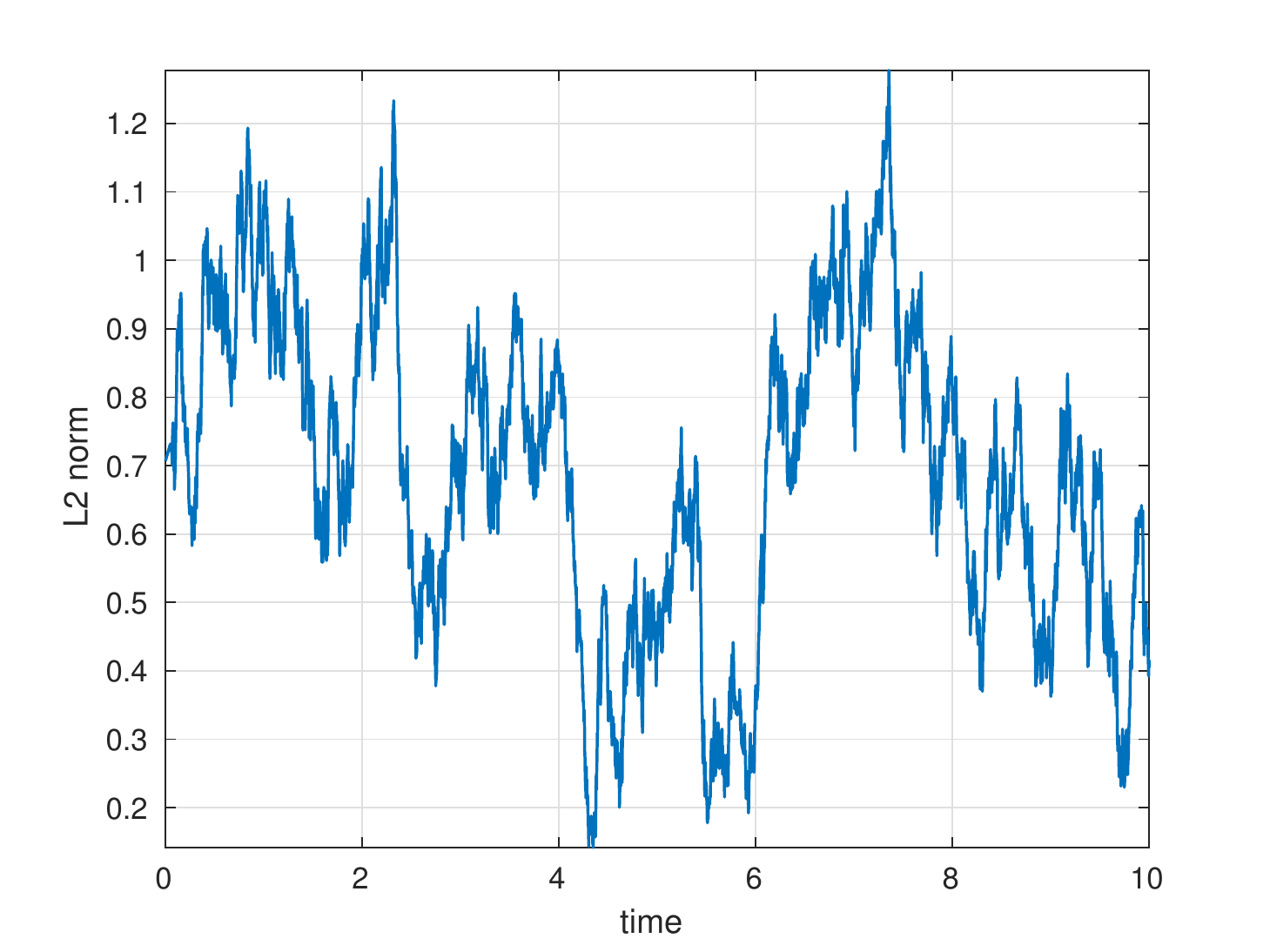}\\
    (b)
    \includegraphics[width=1\textwidth,height=0.15\textheight]{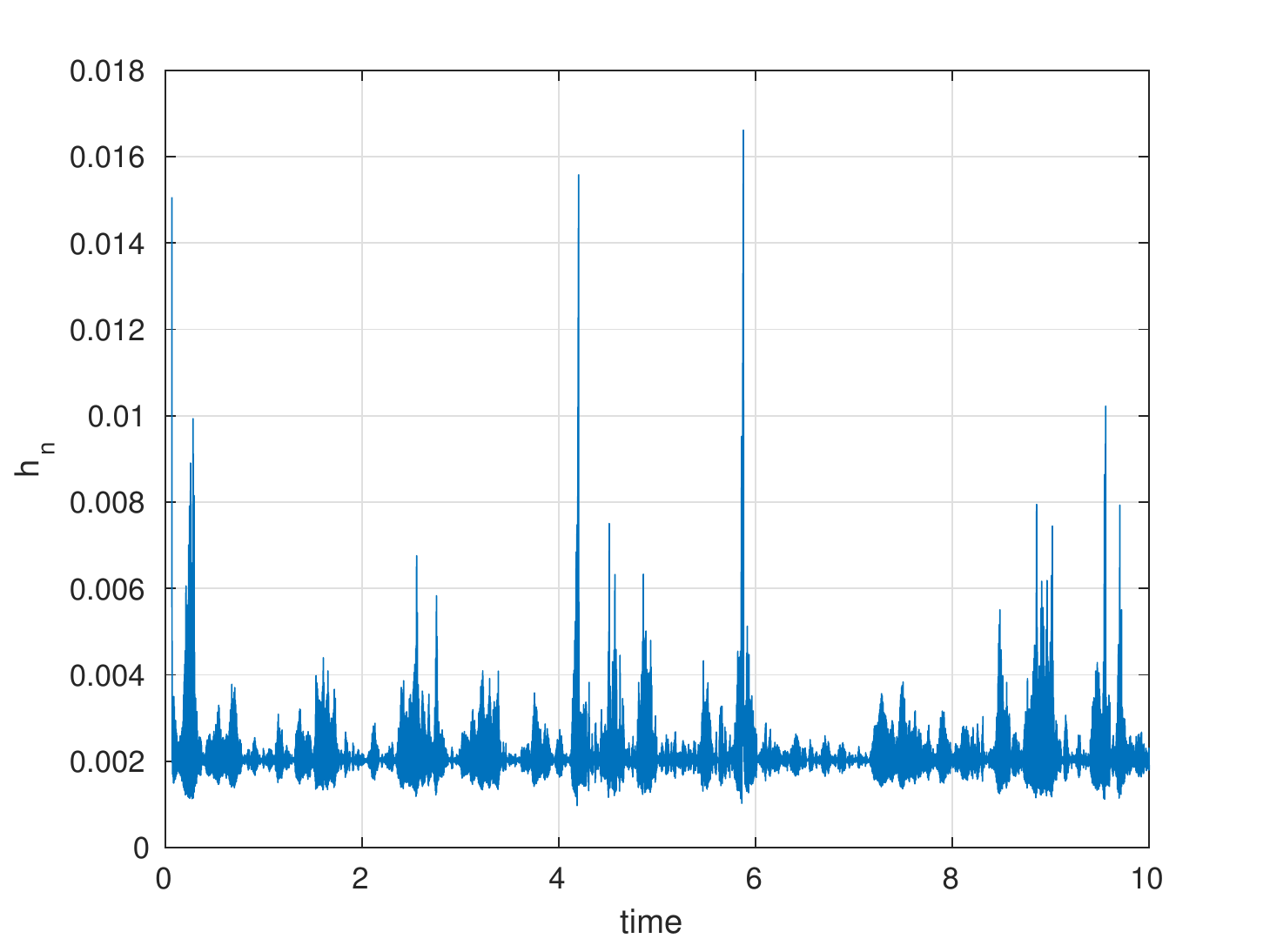}
    \caption{(a) $L^2(0,1)$ norm of a single realisation on the
      stochastic Allen-Cahn SPDE solved using a spectral Galerkin
      system and \EAT. The corresponding timesteps $h_n$ are shown in (b).}
  \end{center}
  \label{fig:ACSpde}
\end{figure}

\subsection{An application to multi-level Monte-Carlo simulation}
One major motivation in \cite{HJ2011} for looking at the
non-convergence of the Euler-Maruyama method was the recent interest
in multi-level Monte-Carlo (MLMC) methods for SDEs, see for example
\cite{Giles15,LPS}.
In its basic form the idea is to use a telescoping sum over different
numerical approximations (levels) as a form of variance reduction.
If we seek to estimate some (Lipschitz) quantity of interest $Q$ of the
solution $X(T)$ to the SDE we can use approximations with a hierarchy
of accuracies from most accurate $L$ to least accurate $0$ and have 
$$\expect{Q(X_{L})} = \expect{Q(X_{L_0})} + \sum_{j=L_0}^{L-1} \expect{Q(X_{j+1})-Q(X_{j})}.$$
We can estimate each expectation on the right hand side with a
different number of realisations determined according to the method described in \cite{Giles15,LPS}, and as $j$ increases we would expect
to take fewer realisations.

We implemented the MLMC method for \EAT and illustrate results below for the
Stochastic Ginzburg-Landau equation with additive noise,
i.e. \eqref{eq:SGLE} with $G(X)=1.$
In our implementation we formed each level by imposing a level
dependent $\hmax^\ell=\hmax^{0}k^{-\ell}$, with $\hmax^0=1$ and
$k=4$. We compare the number of realisations (and hence computational cost) to
those required for the drift-tamed Euler--Maruyama method \eqref{eq:TE}. 
We observe in \cref{fig:mlmcvar} (a) that with the adaptive timestepping 
the variance is reduced at each level compared to taking fixed steps
and hence the number of samples required at leach level is also
reduced (b). This is consistent with other adaptive timestepping
results \cite{HvSST,Giles2017}. 
\begin{figure}
  \begin{center}
    (a) \hspace{0.49\textwidth} (b)
    \includegraphics[width=0.48\textwidth]{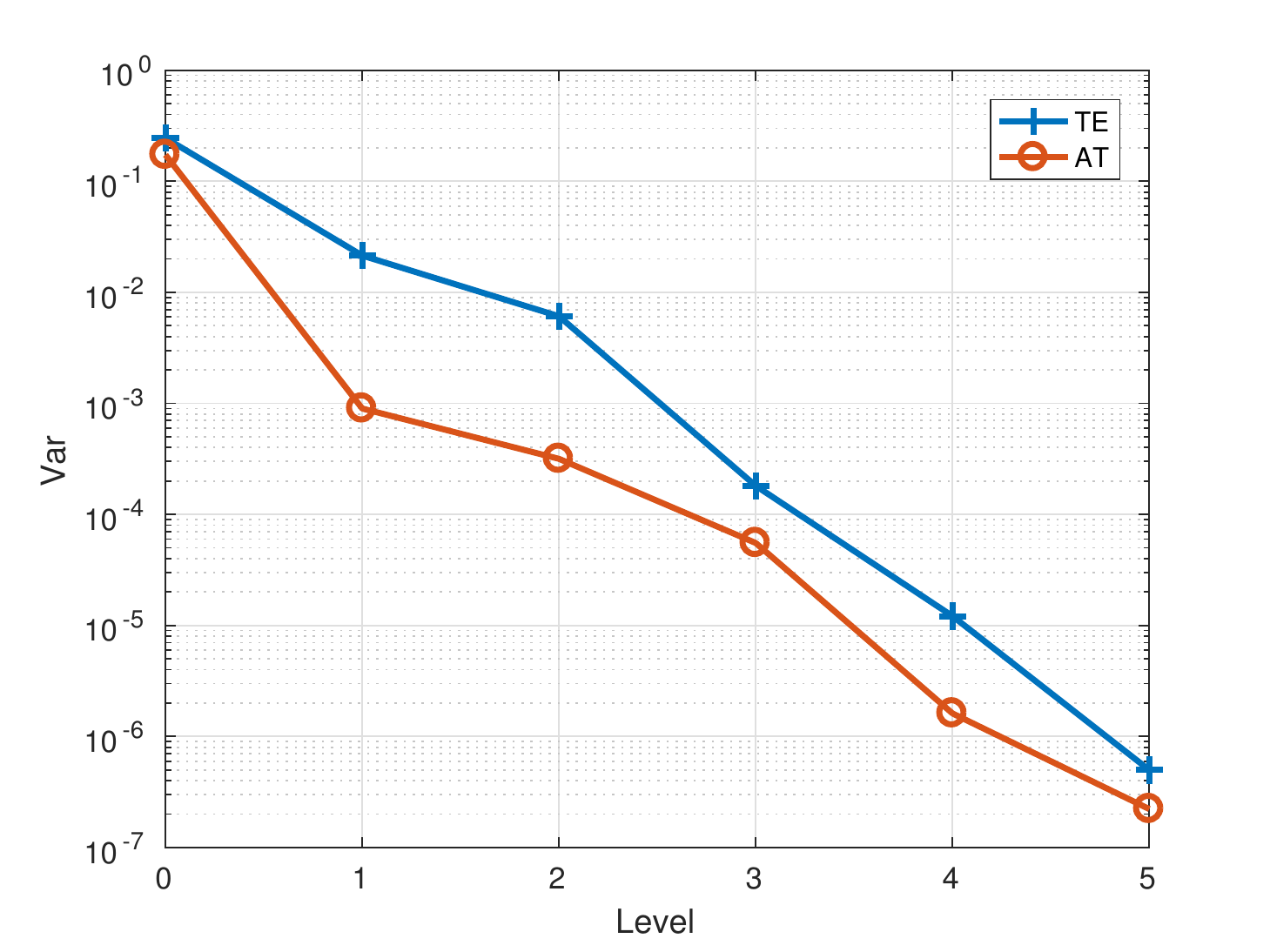} 
    \includegraphics[width=0.48\textwidth]{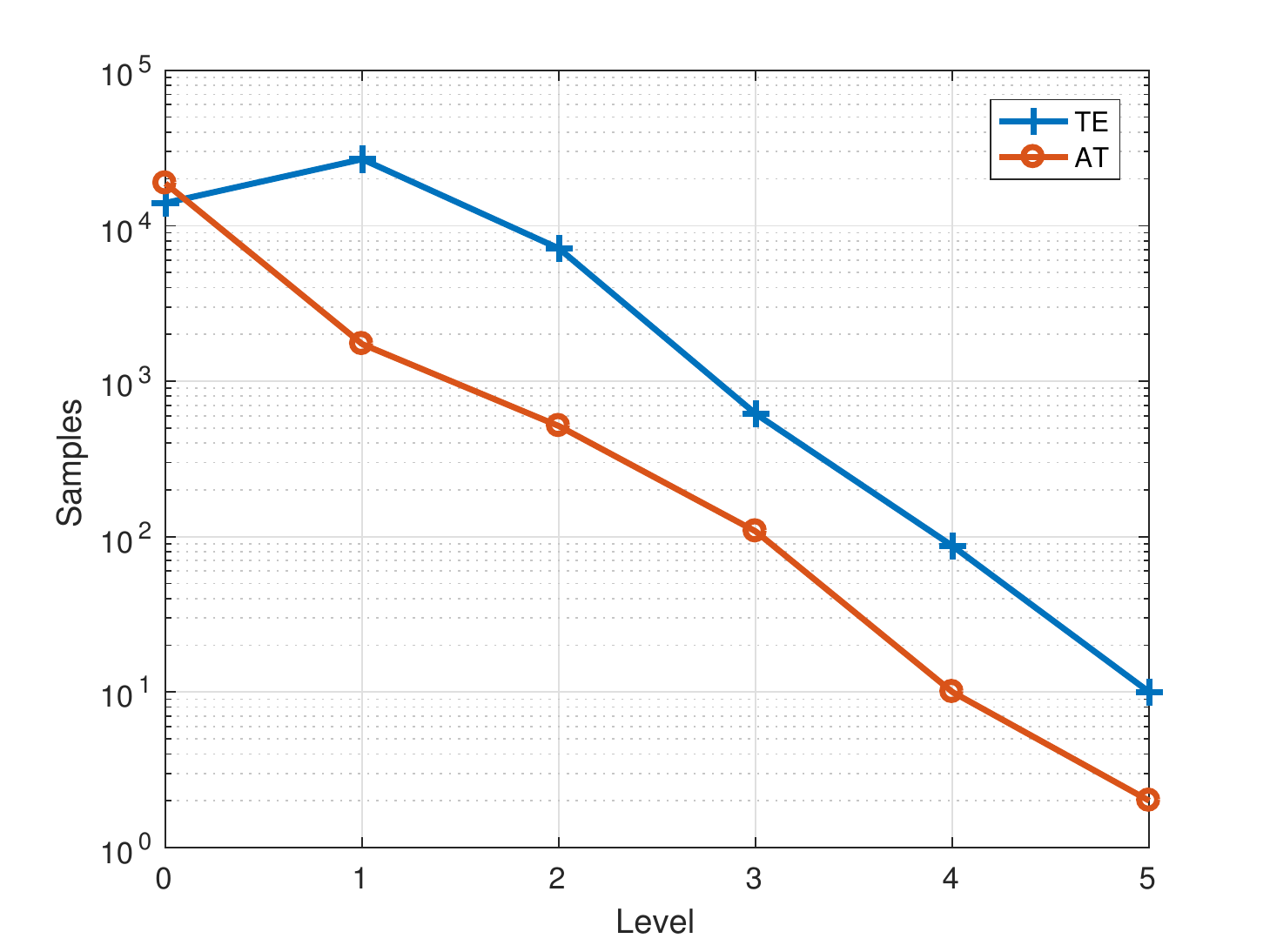} 
  \end{center}
  \caption{Variance against level and number of samples against level
    for \EAT and \eqref{eq:TE} (a). As levels increase the variance
    decreases due to the strong error estimate and hence number of samples
    taken on each level decreases (b). We see that fewer samples are
    required on each level with the adaptive method \EAT.}
  \label{fig:mlmcvar}
\end{figure}

\section{Proof of main result}\label{sec:proofs}
\begin{lemma}
  \label{LPSbook}
Let $(X(t))_{t\in[0,T]}$ be a solution of \eqref{eq:SDE} with initial
value $X(0)=X_0$, and with drift and diffusion coefficients $f$ and $g$
satisfying conditions \eqref{eq:fDer}--\eqref{eq:gLcond}.
Let $\{t_n\}_{n\in\mathbb{N}}$ arise from an adaptive timestepping strategy for \eqref{eq:Scheme} satisfying the conditions of \cref{assum:h}. Consider the Taylor expansions of $f$ and $g$
  \begin{eqnarray*}
    f(X(s))=f(X(t_n))+R_f(s,t_n,X(t_n)),\qquad 
    g(X(s))=g(X(t_n))+R_g(s,t_n,X(t_n)),
  \end{eqnarray*}
  where the remainders $R_f$ and $R_g$ are given in integral form by
  \begin{equation*}
    \label{eq:Rf}
    R_{z}(s,t_n,X(t_n)):= \int_0^1 Dz(X(t_n) + \tau(X(s)-X(t_n))(X(s)-X(t_n)) d\tau,
  \end{equation*}
  and $z$ can be taken to read either $f$ or $g$.
  Then there are a.s. finite and $\mathcal{F}_{t_n}$-measurable random variables $\bar{K}_1,\bar{K}_2>0$, and constants $K_1,K_2,K_3<\infty$, the latter three independent of $h_{n+1}$, such that
\begin{align*}
(i) &\quad
  \expect{\left\|\int_{t_n}^{t_{n+1}}R_{z}(s,t_n,X(t_n))ds\right\|\bigg|\mathcal{F}_{t_n}} \leq \bar{K}_1h_{n+1}^{3/2},\quad a.s.\\
(ii) &\quad
  \expect{\left\|\int_{t_n}^{t_{n+1}}R_{z}(s,t_n,X(t_n))ds\right\|^2\bigg|\mathcal{F}_{t_n}}
 \leq  \bar{K}_2h_{n+1}^2,\quad a.s.\\
(iii) &\quad  \mathbb{E}[\bar{K}_1]\leq K_1, \quad \text{and} \quad\mathbb{E}[\bar{K}_2]\leq K_2.\\
(iv) &\quad
  \expect{\bar{K}_1\left(R_1+2R_2\|X(t_n)\|^2c_1^2\left(1+2\|X(t_n)\|^{c+1}+\|X(t_n)\|^{2(c+1)}\right)\right)}\leq K_3.
\end{align*}
\end{lemma}
\begin{proof}
Let $t_n$ be a term of $\{t_n\}_{n\in\mathbb{N}}$, and suppose that $t_n<s\leq T$. Then
$$X(s)-X(t_n)=\int_{t_n}^s f(X(r))dr+\int_{t_n}^s g(X(r))dW(r).$$
By the triangle inequality, Jensen's inequality, and the conditional form of the It\^o isometry,
\begin{eqnarray*}
\lefteqn{\expect{\|X(s)-X(t_n)\|^2 | \mathcal{F}_{t_n}}}\\ &\leq& 2 \expect{\left\| \int_{t_n}^s f(X(r)) dr\right\|^2 \bigg| \mathcal{F}_{t_n}}+ 2 \expect{\int_{t_n}^s \| g(X(r))\|^2_F dr \bigg| \mathcal{F}_{t_n}}\\
&\leq& 2  \int_{t_n}^s \expect{\left\|f(X(r))\right\|^2 | \mathcal{F}_{t_n}}dr + 2 \int_{t_n}^s \expect{\| g(X(r))\|^2_F | \mathcal{F}_{t_n}} dr,\quad a.s.
\end{eqnarray*}
Next, we apply \eqref{eq:gLcond} and \eqref{eq:fbound} to get
\begin{eqnarray*}
\lefteqn{\expect{\|X(s)-X(t_n)\|^2 | \mathcal{F}_{t_n}}}\\ 
&\leq& 4  \int_{t_n}^s \expect{c_1^2(1+\|X(r)\|^{2c+2}) | \mathcal{F}_{t_n}}dr + 2\kappa^2 \int_{t_n}^s \expect{\| X(r)\|^2 | \mathcal{F}_{t_n}} dr\\
&\leq& \left(4\expect{c_1^2\left(1+\sup_{u\in[0,T]}\|X(u)\|^{2c+2}\right) \bigg| \mathcal{F}_{t_n}}\right.\\
&&\qquad\qquad\left.+2\kappa^2\expect{\sup_{u\in[0,T]}\| X(u)\|^2 \bigg| \mathcal{F}_{t_n}}\right)|s-t_n|\quad a.s.
\end{eqnarray*}
Therefore, by \eqref{eq:SDEmoments} in the statement of \cref{lem:boundedMomentsSDE}, we can define an a.s. finite and $\mathcal{F}_{t_n}$-measurable random variable 
\begin{equation}\label{eq:Ln}
\bar{L}_{n}:=\left(4\expect{c_1^2\left(1+\sup_{u\in[0,T]}\|X(u)\|^{2c+2}\right) \bigg| \mathcal{F}_{t_n}}+2\kappa^2\expect{\sup_{u\in[0,T]}\| X(u)\|^2 \bigg| \mathcal{F}_{t_n}}\right),
\end{equation}
so that 
\begin{equation}\label{eq:holdBound}
\expect{\|X(s)-X(t_n)\|^2\big|\mathcal{F}_{t_n}}\leq \bar{L}_n|s-t_n|,\quad a.s.
\end{equation}

Now consider Part (i) with $R_f$. By \eqref{eq:fDer}, and the Cauchy-Schwarz inequality
\begin{eqnarray*}
\lefteqn{\mathbb{E}\left[\left\|R_f(s,t_n,X(t_n))\right\|\bigg|\mathcal{F}_{t_n}\right]}\\
&\leq& c_1\mathbb{E}\left[\int_0^1 (1+\|X(t_n) +
\tau(X(s)-X(t_n))\|^c)\|(X(s)-X(t_n))\|d\tau\bigg|\mathcal{F}_{t_n}\right]\\
&\leq&c_1\sqrt{\mathbb{E}\left[\|(X(s)-X(t_n))\|^2 | \mathcal{F}_{t_n}\right]}\\
&&\quad\times\sqrt{\mathbb{E}\left[\int_0^1 (1+\|X(t_n) +
\tau(X(s)-X(t_n))\|^c)^2d\tau\bigg|\mathcal{F}_{t_n}\right]},\quad a.s.
\end{eqnarray*}
By \eqref{eq:SDEmoments} in the statement of \cref{lem:boundedMomentsSDE} we can define an a.s. finite and $\mathcal{F}_{t_n}$-measurable random variable
\[
\bar{M}_n:=\expect{2c_1^2+18c_1^2\sup_{u\in[0,T]}\|X(u)\|^{2c}\bigg|\mathcal{F}_{t_n}}
\]
and so, by \eqref{eq:holdBound}, 
\[
\mathbb{E}\left[\|R_f(s,t_n,X(t_n))\|\bigg|\mathcal{F}_{t_n}\right]\leq \sqrt{\bar{M}_n\bar{L}_n}\sqrt{|s-t_n|},\quad a.s.
\]
Since $t_{n+1}$ is an $\mathcal{F}_{t_{n}}$-measurable random variable, there is an a.s. finite and $\mathcal{F}_{t_n}$-measurable random variable $0<\bar{K}_1:=\frac{2}{3}\sqrt{\bar{M}_n\bar{L}_n}$ such that
\begin{eqnarray*}
\expect{\left\|\intdt R_f(s,t_n,X(t_n))ds\right\|\bigg|\mathcal{F}_{t_n}}&\leq&\int_{t_n}^{t_{n+1}}\mathbb{E}\left[\|R_f(s,t_n,X(t_n))\|\mathcal{F}_{t_n}\right]ds\\
&\leq& \sqrt{\bar{M}_n\bar{L}_n} \intdt \sqrt{|s-t_n|} ds \leq \bar{K}_1 h_{n+1}^{3/2},\quad a.s.
\end{eqnarray*}

For Part (i) with $R_g$, the same approach using the global Lipschitz condition \eqref{eq:gLcond} instead of \eqref{eq:osLcond} yields the result. 

Now consider Part (ii) with $R_f$. We have by \eqref{eq:SDEmoments} and the Cauchy-Schwarz inequality 
\begin{eqnarray*}
\lefteqn{\mathbb{E}\left[\left\|\int_{t_n}^{t_{n+1}}R_f(s,t_n,X(t_n))ds\right\|^2\bigg|\mathcal{F}_{t_n}\right]}\\
&\leq& c\mathbb{E}\left[\left\|\int_{t_n}^{t_{n+1}}\int_0^1 Df((X(t_n) +
\tau(X(s)-X(t_n)))(X(s)-X(t_n))d\tau ds\right\|^2\bigg|\mathcal{F}_{t_n}\right]\\
&\leq&c\mathbb{E}\left[\left\|\int_{t_n}^{t_{n+1}}ds\right\|^2\right.\\
&&\qquad\times\left.\sup_{u\in[t_n,t_{n+1}]}\int_0^1 (1+\|X(t_n) + \tau(X(u)-X(t_n))\|^{c})^2\|X(u)-X(t_n)\|^2d\tau\bigg|\mathcal{F}_{t_n}\right]\\
&\leq&c\sqrt{\mathbb{E}\left[\left\|\int_{t_n}^{t_{n+1}}ds\right\|^4\bigg|\mathcal{F}_{t_n}\right]}\\
&&\qquad\times\sqrt{\expect{\left(\sup_{u\in[0,T]}4\int_0^1 (1+(1+2\tau)^c\|X(u)\|^c)^2\|X(u)\|^2d\tau\right)^2\bigg|\mathcal{F}_{t_n}}}\\
&\leq&\bar{K}_2h_{n+1}^2,\quad a.s.,
\end{eqnarray*}
where $\bar{K}_2$ is the a.s. finite and $\mathcal{F}_{t_n}$-measurable random variable
\[
\bar{K}_2:=\sqrt{128c^2\expect{\sup_{u\in[0,T]}\left(\|X(u)\|^4+3^{4c}\|X(u)\|^{4c+4}\right)\bigg|\mathcal{F}_{t_n}}}.
\]
A similar approach for $R_g$ using the global Lipschitz condition \eqref{eq:gLcond} completes Part (ii). 

Part (iii) follows from the construction of $\bar{K}_1$ and $\bar{K}_2$ as follows. An application of Cauchy Schwarz and \eqref{eq:SDEmoments} in the statement of \cref{lem:boundedMomentsSDE} gives that there exists $K_1<\infty$, independent of $h_{n+1}$, such that
\begin{eqnarray*}
\expect{\bar{K}_1}&=&\expect{\frac{2}{3}\sqrt{\bar{M}_n}\sqrt{\bar{L}_n}}\leq \frac{2}{3}\sqrt{\expect{\bar{M}_n}}\sqrt{\expect{\bar{L}_n}}=: K_1.
\end{eqnarray*}
A similar argument using Jensen's inequality shows that there exists $K_2<\infty$, independent of $h_{n+1}$, such that $\expect{\bar{K}_2}\leq K_2$.

Finally, for Part (iv), define the a.s. finite and $\mathcal{F}_{t_n}$-measurable random variable
\[
P(\|X(t_n)\|):=R_1+2R_2\|X(t_n)\|^2c_1^2\left(1+2\|X(t_n)\|^{c+1}+\|X(t_n)\|^{2(c+1)}\right).
\]
Then, by Cauchy-Schwarz and \eqref{eq:SDEmoments} in the statement of \cref{lem:boundedMomentsSDE}, we have that there exists $K_3<\infty$, independent of $h_{n+1}$, such that
\[
\expect{\bar{K}_1P(\|X(t_n))\|}=\expect{\sqrt{\bar{M}_n}\sqrt{\bar{L}_nP(\|X(t_n)\|)^2}}\leq
\expect{\sqrt{\bar{M}_n}\sqrt{\bar{L}_n^{\ast}}}=K_3,
\]
where, by \eqref{eq:Ln}, we have
\begin{multline*}
\bar{L}_n^{\ast}=4\expect{c_1^2\sup_{u\in[0,T]}\left(\left(1+\|X(u)\|^{2c+2}\right)P(\|X(u)\|)^2\right) \bigg| \mathcal{F}_{t_n}}\\+2\kappa^2\expect{\sup_{u\in[0,T]}\left(\| X(u)\|^2P(\|X(u)\|)^2\right) \bigg| \mathcal{F}_{t_n}}.
\end{multline*}
This completes the proof.
\end{proof}

\begin{proof}[Proof of \cref{thm:adaptConv}]
By \cref{thm:tamedConv} it is sufficient to consider only the event that $h_{\text{min}}<h_{n}<h_{\text{max}}$ for all $n=0,\ldots,N-1$. Define the error sequence $\{E_n\}_{n\in\mathbb{N}}$ by
\begin{eqnarray*}
E_{n+1}&:=&Y_{n+1}-X(t_{n+1})\\
&=&Y_n-X(t_n)+\int_{t_n}^{t_{n+1}}[f(Y_n)-f(X(s))]ds+\int_{t_n}^{t_{n+1}}[g(Y_n)-g(X(s))]dW(s).
\end{eqnarray*}
Expand $f$ and $g$ as Taylor series around $X(t_n)$ over the interval of
integration. As in \cref{LPSbook} we get 
\begin{multline*}
  E_{n+1}=E_n+\int_{t_n}^{t_{n+1}}[f(Y_n)-f(X(t_n))]ds+\int_{t_n}^{t_{n+1}}[g(Y_n)-g(X(t_n))]dW(s)\\
  +\underbrace{\int_{t_n}^{t_{n+1}}R_f(s,t_n,X(t_n))ds}_{:=\tilde{R}_f(t_n,X(t_n))}+\underbrace{\int_{t_n}^{t_{n+1}}R_g(s,t_n,X(t_n))dW(s)}_{:=\tilde{R}_g(t_n,X(t_n))}
\end{multline*}
which may be rewritten using the notation $\triangle W_{n+1}=W(t_{n+1})-W(t_n)$ as
\begin{multline*}
E_{n+1}=E_n+h_{n+1}[f(Y_n)-f(X(t_n))]+\triangle W_{n+1}[g(Y_n)-g(X(t_n))]\\+\tilde{R}_f(t_n,X(t_n))+\tilde{R}_g(t_n,X(t_n)).
\end{multline*}
Next we develop appropriate bounds on
\begin{equation}\label{eq:condExp}
  \mathbb{E}\left[\|E_{n+1}\|^2\big|\mathcal{F}_{t_n}\right]=\mathbb{E}\left[\langle E_{n+1},E_{n+1}\rangle\big|\mathcal{F}_{t_n}\right].
\end{equation}
Note that
\begin{eqnarray*}
\lefteqn{\|E_{n+1}\|^2=\langle E_n,E_{n+1}\rangle +\underbrace{h_{n+1}\langle f(Y_n)-f(X(t_n)),E_{n+1}\rangle}_{:=A_n}}&&\\
&& +\underbrace{\langle \triangle W_{n+1}[g(Y_n)-g(X(t_n))],E_{n+1}\rangle}_{:=B_n}+\underbrace{\langle \tilde{R}_f(t_n,X(t_n))+\tilde{R}_g(t_n,X(t_n)),E_{n+1}\rangle}_{C_n}.
\end{eqnarray*}
Then, since $\langle E_n,E_{n+1}\rangle\leq\frac{1}{2}(\|E_n\|^2+\|E_{n+1}\|^2)$, we have $
\|E_{n+1}\|^2=\|E_n\|^2+2A_n+2B_n+2C_n$. Next we omit the arguments from $\tilde{R}_f$, $\tilde{R}_g$ and write
\begin{eqnarray*}
A_n&=&h_{n+1}\langle f(Y_n)-f(X(t_n)),E_n\rangle+h_{n+1}^2\|f(Y_n)-f(X(t_n))\|^2\\
&&+h_{n+1}\langle f(Y_n)-f(X(t_n)),\triangle W_{n+1}[g(Y_n)-g(X(t_n))]\rangle\\
&&+h_{n+1}\langle f(Y_n)-f(X(t_n)),\tilde{R}_f+\tilde{R}_g\rangle;\\
B_n&=&\langle \triangle W_{n+1}[g(Y_n)-g(X(t_n))],E_n\rangle+\|\triangle W_{n+1}[g(Y_n)-g(X(t_n))]\|^2\\
&&+h_{n+1}\langle \triangle W_{n+1}[g(Y_n)-g(X(t_n))],f(Y_n)-f(X(t_n))\rangle\\
&&+\langle \triangle W_{n+1}[g(Y_n)-g(X(t_n))],\tilde{R}_f+\tilde{R}_g\rangle;\\
C_n&=&\langle \tilde{R}_f+\tilde{R}_g, E_n\rangle+\|\tilde{R}_f+\tilde{R}_g\|^2\\
&&+h_{n+1}\langle \tilde{R}_f+\tilde{R}_g,f(Y_n)-f(X(t_n))\rangle\\
&&+\langle \triangle W_{n+1}[g(Y_n)-g(X(t_n))],\tilde{R}_f+\tilde{R}_g\rangle.
\end{eqnarray*}
By \cref{rem:conditionalMoments}, and applying the Lipschitz bounds \eqref{eq:osLcond} and \eqref{eq:gLcond}, we may now estimate \eqref{eq:condExp} as
\begin{eqnarray*}
\lefteqn{\mathbb{E}\left[\|E_{n+1}\|^2\big|\mathcal{F}_{t_n}\right]}\\
&\leq&\|E_{n}\|^2+h_{n+1}(2\alpha+2\kappa^2)\|E_{n}\|^2+2h_{n+1}^2\|f(Y_n)-f(X(t_n))\|^2\\
&&+\underbrace{4h_{n+1}\mathbb{E}\left[\langle f(Y_n)-f(X(t_n)),\tilde{R}_f+\tilde{R}_g\rangle\big|\mathcal{F}_{t_n}\right]}_{:=\bar{A}_n}\\
&&+\underbrace{4\mathbb{E}\left[\langle \triangle W_{n+1}[g(Y_n)-g(X(t_n))],\tilde{R}_f+\tilde{R}_g\rangle\big|\mathcal{F}_{t_n}\right]}_{:=\bar{B}_n}\\
&&+\underbrace{2\mathbb{E}\left[\langle \tilde{R}_f+\tilde{R}_g,E_n\rangle\big|\mathcal{F}_{t_n}\right]}_{:=\bar{C}_n}+\underbrace{2\mathbb{E}\left[\|\tilde{R}_f+\tilde{R}_g\|^2\big|\mathcal{F}_{t_n}\right]}_{:=\bar{D}_n},\quad a.s.
\end{eqnarray*}
Since $\{t_n\}_{n\in\mathbb{N}}$ arises from an admissible timestepping strategy, we can use \eqref{eq:normfBound} along with the bound on \eqref{eq:fbound} on $f$ to get
\begin{eqnarray}
\lefteqn{h_{n+1}^2\|f(Y_n)-f(X(t_n))\|^2}\nonumber\\&\leq& 2h_{n+1}^2\|f(Y_n)\|^2+2h_{n+1}^2\|f(X(t_n))\|^2\nonumber\\
&\leq&2h_{n+1}^2R_2\|Y_n\|^2+2h_{n+1}^2R_1+2h_{n+1}^2\|f(X(t_n))\|^2\nonumber\\
&\leq&4h_{n+1}^2R_2\|E_n\|^2+4h_{n+1}^2R_2\|X(t_n)\|^2+2h_{n+1}^2R_1+2h_{n+1}^2\|f(X(t_n))\|^2\nonumber\\
&\leq&4h_{n+1}^2R_2\|E_n\|^2+2h_{n+1}^2R_1\nonumber\\
&&+4h_{n+1}^2R_2\|X(t_n)\|^2+2c_1^2h_{n+1}^2\left(1+2\|X(t_n)\|^{c+1}+\|X(t_n)\|^{2(c+1)}\right).\label{eq:barAbound}
\end{eqnarray}
Now we can write
\begin{multline}\label{eq:beforeSumming}
\mathbb{E}\left[\|E_{n+1}\|^2\big|\mathcal{F}_{t_n}\right]-\|E_n\|^2\leq h_{n+1}(2\alpha+2\kappa^2+8R_2)\|E_{n}\|^2\\
+4h_{n+1}^2\left[R_1+2R_2\|X(t_n)\|^2+c_1^2\left(1+2\|X(t_n)\|^{c+1}+\|X(t_n)\|^{2(c+1)}\right)\right]\\
+\bar{A}_n+\bar{B}_n+\bar{C}_n+\bar{D}_n,\quad a.s.
\end{multline}
Next we must consider the terms $\bar{A}_n$, $\bar{B}_n$, $\bar{C}_n$, and $\bar{D}_n$. 
After an application of the triangle inequality, it immediately
follows from Part (ii) of \cref{LPSbook} that 
\[
\bar{D}_n=2\mathbb{E}\left[\|\tilde{R}_f+\tilde{R}_g\|^2\big|\mathcal{F}_{t_n}\right]\leq 8\bar{K}_2h_{n+1}^2,\quad a.s.
\]
This estimate, along with the conditional second moment of $\triangle W_{n+1}$ provided in \cref{rem:conditionalMoments}, and additionally applying two variants of the Cauchy-Schwarz inequality, first to the inner product and second to the conditional expectation, gives us
\begin{eqnarray*}
\bar{B}_n&=&4\mathbb{E}\left[\langle \triangle W_{n+1}[g(Y_n)-g(X(t_n))],\tilde{R}_f+\tilde{R}_g\rangle\big|\mathcal{F}_{t_n}\right]\\
&\leq&\kappa\|E_n\|\mathbb{E}\left[\langle\triangle W_{n+1},\tilde{R}_f+\tilde{R}_g\rangle\big|\mathcal{F}_{t_n}\right]\\
&\leq&\kappa\|E_n\|\mathbb{E}\left[\|\triangle W_{n+1}\|\|\tilde{R}_f+\tilde{R}_g\||\mathcal{F}_{t_n}\right]\\
&\leq&\kappa\|E_n\|\sqrt{\expect{\|\triangle W_{n+1}\|^2\big|\mathcal{F}_{t_n}}}\sqrt{\expect{\|\tilde{R}_f+\tilde{R}_g\|^2\big|\mathcal{F}_{t_n}}}\\
&\leq&2\kappa\sqrt{\bar{K}_2}\|E_n\|h_{n+1}^{3/2}\\
&\leq&\frac{1}{2}\|E_n\|^2h_{n+1}+\kappa^2\bar{K}_2h_{n+1}^2,\quad a.s.
\end{eqnarray*}
Part (i) of \cref{LPSbook} yields
\begin{eqnarray*}
\bar{C}_n&=&2\mathbb{E}\left[\langle \tilde{R}_f+\tilde{R}_g,E_n\rangle\big|\mathcal{F}_{t_n}\right]\\
&\leq&2\expect{\|\tilde{R}_f\|\|E_n\|\big|\mathcal{F}_{t_n}}+2\expect{\|\tilde{R}_g\|\|E_n\|\big|\mathcal{F}_{t_n}}\\
&\leq&4\bar{K}_1\|E_n\|h_{n+1}^{3/2}\leq 2\bar{K}_1h_{n+1}\|E_n\|^2+2\bar{K}_1h_{n+1}^2,\quad a.s.
\end{eqnarray*}
Finally,
\begin{eqnarray*}
\bar{A}_n&=&4h_{n+1}\mathbb{E}\left[\langle f(Y_n)-f(X(t_n)),\tilde{R}_f+\tilde{R}_g\rangle\bigg|\mathcal{F}_{t_n}\right]\\
&=&4h_{n+1}\mathbb{E}\left[\langle f(Y_n)-f(X(t_n)),\tilde{R}_f\rangle\bigg|\mathcal{F}_{t_n}\right]+4h_{n+1}\mathbb{E}\left[\langle f(Y_n)-f(X(t_n)),\tilde{R}_g\rangle\bigg|\mathcal{F}_{t_n}\right].
\end{eqnarray*}
Moreover we have
\begin{eqnarray*}
\mathbb{E}\left[\langle f(Y_n)-f(X(t_n)),\tilde{R}_f\rangle\bigg|\mathcal{F}_{t_n}\right]&\leq&\mathbb{E}\left[\|f(Y_n)-f(X(t_n))\|\|\tilde{R}_f\|\bigg|\mathcal{F}_{t_n}\right]\\
&=&\|f(Y_n)-f(X(t_n))\|\expect{\|\tilde{R}_f\||\mathcal{F}_{t_n}}\quad a.s.
\end{eqnarray*}
A similar bound holds for $\mathbb{E}\left[\langle f(Y_n)-f(X(t_n)),\tilde{R}_g\rangle\bigg|\mathcal{F}_{t_n}\right]$, and therefore,  by Part (i) of \cref{LPSbook},
\begin{eqnarray*}
\bar{A}_n&\leq& 8\|f(Y_n)-f(X(t_n))\|\bar{K}_1h_{n+1}^{5/2}\nonumber\\
&\leq&4\bar{K}_1h_{n+1}^2+4\bar{K}_1\|f(Y_n)-f(X(t_n))\|^2h_{n+1}^3,\quad a.s.\label{eq:barAupper}
\end{eqnarray*}
Applying these bounds to \eqref{eq:beforeSumming}, along with \eqref{eq:barAbound} and noting that $h_{\text{max}}\leq 1$, yields
\begin{multline}\label{eq:beforeSumming2}
\mathbb{E}\left[\|E_{n+1}\|^2\big|\mathcal{F}_{t_n}\right]-\|E_n\|^2\leq h_{n+1}\Gamma_2\|E_{n}\|^2\\+\left[6\bar{K}_1+(8+\kappa^2)\bar{K}_2+4(1+\bar{K}_1)\left[R_1+2R_2\|X(t_n)\|^2\right.\right.\\
\left.\left.+c_1^2\left(1+2\|X(t_n)\|^{c+1}+\|X(t_n)\|^{2(c+1)}\right)\right]\right]h_{n+1}^2,\quad a.s.
\end{multline}
where, recalling $K_1$ as defined in Part (iii) of the statement of \cref{LPSbook}, 
\[
\Gamma_2:=2(\alpha+\kappa^2)+1/2+2K_1+24R_2.
\]
Summing both sides of \eqref{eq:beforeSumming2} over $n$ from $0$ to $N-1$ and taking expectations yields
\begin{equation}\label{eq:preGronwall}
\mathbb{E}[\|E_N\|^2]\leq \Gamma_1T h_{\text{max}}+ \Gamma_2h_{\text{max}}\sum_{n=0}^{N-1}\mathbb{E}[\|E_{n}\|^2],
\end{equation}
where, recalling $K_2$ and $K_3$ as defined in Parts (iii) and (iv) of
the statement of \cref{LPSbook}, we have defined the constant $\Gamma_1$ as
\begin{multline*}
\Gamma_1:=6K_1+(8+\kappa^2)K_2+4\left[R_1+2R_2\expect{\sup_{u\in[0,T]}\|X(u)\|^2}\right.\\
\left.+c_1^2\left(1+2\mathbb{E}\left[\sup_{u\in[0,T]}\|X(u)\|^{c+1}\right]+\mathbb{E}\left[\sup_{u\in[0,T]}\|X(u)\|^{2(c+1)}\right]\right)\right]+4K_3.
\end{multline*}

The discrete Gronwall inequality (see for example \cite{StuartHumpries}), \eqref{eq:hRatio}, and the fact that $Nh_{\text{min}}\leq T$, may now be applied to \eqref{eq:preGronwall}:
\begin{eqnarray*}
\mathbb{E}[\|E_N\|^2]&\leq&  \Gamma_1Th_{\text{max}}\exp\left\{Nh_{\text{max}}\Gamma_2\right\}\\
&=&   \Gamma_1Th_{\text{max}}\exp\left\{\rho Nh_{\text{min}}\Gamma_2\right\}\\
&\leq & \Gamma_1Th_{\text{max}}\exp\left\{\rho T\Gamma_2\right\},
\end{eqnarray*}
which gives the statement of the theorem.
\end{proof}

\section{Conclusions and future work}\label{sec:concl}
We introduced a class of adaptive timestepping strategies for
SDEs with non-Lipschitz drift coefficients and proved strong convergence without the need to prove additional moment bounds on the numerical method. 

Our numerical results on the stochastic Van der Pol equation indicate
that adaptive timestepping may lead to dynamically more accurate
solutions than those from a fixed step tamed scheme where the drift is
perturbed, this was not noly true for the two adaptive schemes we
presentde here. From the suite of problems we examined the method
\EALD seems to lead to a smaller variance in the timestep selection. 
We also saw from the numerical experiments that the parameter $\rho$
required in the analysis is not a restriction. 
Adaptive timestepping strategies are readily applicable to large scale systems, such as
the Allen-Cahn SPDE. We also saw that when applied in a MLMC context
adaptivity can lead to more efficient computation. It has already been noted by a number of
authors \cite{Sotiropoulos2008,LMS2006,beskos2013} that adaptivity
maybe useful for Langevin sampling dynamics and our analysis offers 
techniques suitable for equations with non-Lipschitz drift terms that
may arise, for example, in image processing. We also note that this approach could be combined with error control timestepping strategies. 

Possible future work includes extending the analysis to include to
SDEs with non-Lipschitz diffusion coefficients, to SDEs with L\'evy
noise, to SPDEs and to other forms of explicit methods.

\section{Acknowledgement}
The authors thank Prof. Alexandra Rodkina (UWI) for her feedback on an
earlier draft.

\bibliographystyle{plain}
\bibliography{bibliog}

\end{document}

%% file: adaptive_shared.tex
\usepackage{lipsum}
\usepackage{amsfonts}
\usepackage{graphicx}
\usepackage{epstopdf}
\usepackage{algorithmic}
\usepackage{amssymb}

\usepackage{amsmath}
\usepackage{latexsym}
\usepackage{epsfig}
\usepackage{xspace}
\usepackage[mathscr]{eucal}
\usepackage{color}
\usepackage{cite}
\usepackage{enumitem,kantlipsum}

\ifpdf
  \DeclareGraphicsExtensions{.eps,.pdf,.png,.jpg}
\else
  \DeclareGraphicsExtensions{.eps}
\fi

\newcommand{\TheTitle}{Adaptive timestepping strategies for nonlinear stochastic systems} 
\newcommand{\TheShortTitle}{Adaptive timestepping for SDEs}
\newcommand{\TheAuthors}{C\'onall Kelly and Gabriel J. Lord}

\headers{\TheShortTitle}{\TheAuthors}

\title{{\TheTitle}\thanks{Submitted to the editors 
October 12, 2016.}}

\author{C\'onall Kelly
  \thanks{{Department of Mathematics, The University of the West Indies, Mona, Kingston 7, Jamaica.}
    (\email{conall.kelly@uwimona.edu.jm}, \url{}), \funding{Supported by a SQuaRE activity entitled ``Stochastic stabilisation of limit-cycle dynamics in ecology and neuroscience'' funded by the American Institute of Mathematics.}.}
  \and
  Gabriel J. Lord \thanks{Maxwell Institute, Department of
    Mathematics, MACS, Heriot-Watt University, Edinburgh, EH14 4AS,
    UK.(\email{g.j.lord@hw.ac.uk}, \url{www.macs.hw.ac.uk/\~gabriel}).}
}

\usepackage{amsopn}
\newtheorem{assumption}[theorem]{Assumption}
\newtheorem{remark}[theorem]{Remark}

\newcommand{\real}{\mathbb{R}}
\newcommand{\hmin}{h_{\min}}
\newcommand{\hmax}{h_{\max}}
\newcommand{\hmean}{h_{\text{mean}}}

\newcommand{\intdt}{\int_{t_n}^{t_{n+1}}}

\newcommand{\expect}[1]{\mathbb{E}\left[  {#1} \right]}

\newcommand{\EAT}{AT\xspace}
\newcommand{\EALD}{ALD\xspace}
\newcommand{\EFT}{FT\xspace}
\newcommand{\EFLD}{FLD\xspace}

\newcommand{\secref}[1]{Section \ref{#1}\xspace}

\newcommand{\Ito}{It\^o\xspace}
